\documentclass[a4paper, 11pt]{article}

\usepackage[utf8]{inputenc}
\usepackage[T1]{fontenc}
\usepackage{fullpage}

\usepackage{comment}
\usepackage{amsmath, amsthm, amssymb}
\usepackage{graphicx}
\usepackage{enumerate}
\usepackage{authblk}
\usepackage{thmtools}
\usepackage{thm-restate}
\usepackage[colorlinks=true, citecolor=red]{hyperref}
\usepackage{float}

\usepackage[linesnumbered,ruled,onelanguage,vlined]{algorithm2e}

\usepackage{tikz}
\renewenvironment{abstract}
{\small\vspace{-1em}
\begin{center}
\bfseries\abstractname\vspace{-.5em}\vspace{0pt}
\end{center}
\list{}{
\setlength{\leftmargin}{0.6in}%
\setlength{\rightmargin}{\leftmargin}}%
\item\relax}
{\endlist}
\usepackage{comment}
\declaretheorem[name=Theorem, numberwithin=section]{theorem}
\declaretheorem[name=Lemma, sibling=theorem]{lemma}

\declaretheorem[name=Conjecture, sibling=theorem]{conjecture}

\declaretheorem[name=Claim, sibling=theorem]{claim}

\def\cqedsymbol{\ifmmode$\lrcorner$\else{\unskip\nobreak\hfil
\penalty50\hskip1em\null\nobreak\hfil$\lrcorner$
\parfillskip=0pt\finalhyphendemerits=0\endgraf}\fi}

\def\A{\mathcal{A}} 
\usetikzlibrary{calc}
\let\le\leqslant
\let\ge\geqslant

\let\geq\geqslant

\title{Square coloring planar graphs with automatic discharging\thanks{This work was supported by ANR project GrR (ANR-18-CE40-0032)}}

\author[2]{Nicolas Bousquet}
\author[2]{Quentin Deschamps}
\author[1]{Lucas de Meyer}
\author[2]{Théo Pierron}

\affil[1]{Département Informatique, ENS Rennes}
\affil[2]{Univ. Lyon, Université Lyon 1, LIRIS UMR CNRS 5205, F-69621, Lyon, France}

\date{\today}

\begin{document}

\maketitle

\begin{abstract}
    The discharging method is a powerful proof technique, especially for graph coloring problems. Its major downside is that it often requires lengthy case analyses, which are sometimes given to a computer for verification. However, it is much less common to use a computer to actively look for a discharging proof. In this paper, we use a Linear Programming approach to automatically look for a discharging proof. While our system is not entirely autonomous, we manage to make some progress towards Wegner's conjecture for distance-$2$ coloring of planar graphs, by showing that $12$ colors are sufficient to color at distance $2$ every planar graph with maximum degree~$4$.
\end{abstract}

\section{Introduction}

The Four Color Theorem, stating that every planar map can be colored with $4$ colors in such a way that no adjacent regions receive similar colors, is famous for having one of the first computer assisted proofs. Even if this breakthrough from~\cite{appel1976every} encountered some philosophical opposition about the nature of proofs, using computers to prove mathematical statements became more and more popular in the last decades. In particular, many researchers followed the steps of Appel and Haken and used computed-aided proofs with the so-called \emph{discharging method}, but also in various fields of mathematics and computer science such as complex analysis~\cite{lanford1982computer}, geometry~\cite{brakensiek2020,hales2017,lam_thiel_swiercz_1989}, game theory~\cite{rokicki2010god}, logic~\cite{dahn1998robbins} and, combinatorics~\cite{heule2016solving,heule2018schur,mcguire2014there,szekeres_peters_2006} just to cite a few. 

Most of these results make great use of the ability of computers to perform efficiently some computations, in particular exhaustive search and checks, much quicker than a human (and probably more reliably). Therefore, it is not surprising that many computer assisted proofs rely on such mechanics (either directly or using arguments that permit to reduce the proof to the study of finitely many cases), see~\cite{appel1976every,heule2016solving,lam_thiel_swiercz_1989,mcguire2014there,rokicki2010god,szekeres_peters_2006}. Broadly speaking, for these results, the authors prove that the statement can be reduced to checking finitely many cases and then design an \emph{ad hoc} program that goes through all the required cases (often using symmetries to consider only a reasonable amount of cases).

Another way to use computers to build proofs consists in using enumeration, whose goal is to generate efficiently all possible structures (such as graphs) with specific properties. 
On a first hand, it permits to test some conjectures on small graphs and to discard many false statements, see e.g.~\cite{brinkmann2013generation,brinkmann:hal-00990486}.  On the other hand, one can also ask the computer to generate graphs with specific properties, which can be then used by the human as gadgets to build counterexamples (see e.g.~\cite{cohen2017steinberg}) or hardness reductions (see e.g.~\cite{mulzer2008minimum}).
Again, in all these examples, most of the proof is still done by the human, who has to find the desirable properties, identify good patterns in the graphs output by the computer, and combine them in a suitable way. The computer is thus limited to an oracle whose power is to output quite naively some constructions, without strong warranty that they will be of any use. 
Some other ways to use computers have been introduced, by translating the original problem into instances of other problems such as SAT~\cite{brakensiek2020,heule2018schur}, linear programming~\cite{hales2017} or semi-definite programming~\cite{grzesik2015flag,silva2016flag}.
A last family of computer assisted proofs contains the ones obtained using proof assistants such as Coq. In that case, the computer work consists only in helping to check step by step a mathematical proof, and to make the writing of such a proof more comfortable.

\paragraph{Automatic discharging proofs.}
Our goal in this paper is to obtain discharging proofs using a computer. The discharging method is a very useful tool in graph theory, especially to prove coloring results on sparse classes of graphs. While powerful, this method has a major downside of often requiring lengthy case analyses. Therefore, it is not surprising that computers are used to automatically go through all the cases, see for example~\cite{hartke2016chromatic,bonamy2019every}. However, in all these cases, the computer only has a passive role: the formalization of the problem and the structure of the proof is provided by the human while the computer is only used to check/generate all the cases. The examples where the human gives the statement to prove and where the machine then plays an active role in the proof construction are, as far as we know, much rarer. One can think of~\cite{dahn1998robbins}, where the computer is used to search for a proof by generating consequences of the axioms until it reaches the sought conclusion. 

Regarding the discharging method, we would like to obtain a system that, after a small modelization step performed by the human, would be able to completely generate the proof by itself, instead of just checking some case analysis. Such a system has already been used, for example in~\cite{dvovrak2021cyclic,hebdige2016third}, but the underlying generative programs are not formally discussed nor published (only checkers are provided). In this paper, our main contribution consists in presenting such a program, that uses Linear Programming to construct a discharging argument and making our code available online.  In particular our program gives an automatic way to find discharging rules, which is left open in~\cite{talon2020intensive} and considered as ``\textit{a big step forward in the field of automatic proofs for planar graphs}''. Moreover, our approach is generic, in the sense that it can be reused for many types of problems or graph classes.

A formal description of discharging proofs and our method is provided in Section~\ref{sec:overview}. From a high level perspective, our program will, at every step, either find a set of discharging rules ensuring that the result holds or output a set of substructures called configurations. When the second option occurs, the human (with the help of a second independent algorithm in this work) has to prove that one of configurations can be ``reduced'' (this will be formally defined later). We can then re-use our system by taking into account the reducibility of new configurations. We repeat this back and forth process until the system stabilizes. When stabilization is reached, either we get a full discharging proof, or we cannot reduce any output configuration (and we have to rethink our modeling).

A rather similar system has been used in another context in~\cite{stolee2014automated} for studying identifying codes in grids, except that the process is one-shot: the author simply model his problem into a Linear Program and does not use a back-and-forth procedure.

In~\cite{la2022computer}, the authors very recently provided a computer assisted discharging proof. Their program allows, once a set of rules is given, to generate the configurations to reduce. The main difference with our process is that determining a suitable set of rules is done by the human in their case while this step is automatic in our approach. 

While our system is not yet completely autonomous (since it needs the intervention of a human at the end of every round of our algorithm), we believe it is a first step towards a fully automated discharging prover, that would be able (with enough computing power) to reprove many discharging results from the literature. We discuss possible further works to reach that goal in conclusion.

\paragraph{Our results.}
As a proof of concept, we use our system to improve the best known result of the conjecture of Wegner on distance-$2$ chromatic number of \emph{planar graphs} of maximum degree $4$. A \emph{distance-$2$ coloring} of a graph $G$ is an assignment of colors to the vertices of $G$ such that no two vertices within distance at most $2$ receive the same color. The minimum number of colors required for such a coloring to exist is denoted by $\chi_2(G)$. This is equivalent to a (standard) coloring of the graph $G^2$, obtained from $G$ by adding edges between vertices with a common neighbor. Wegner's conjecture states an upper bound on $\chi_2(G)$ depending on the maximum degree $\Delta(G)$. 

\begin{conjecture}[Wegner~\cite{wegner1977graphs}]
Every planar graph G with maximum degree $\Delta$ satisfies:
$$
\chi_2(G) \le \left\{
    \begin{array}{ll}
        7 & \mbox{if } \Delta = 3, \\
        \Delta + 5 & \mbox{if } 4 \le \Delta \le 7, \\
        \lfloor \frac{3\Delta}{2} \rfloor  + 1 & \mbox{if } \Delta \ge 8.
    \end{array}
\right.
$$
\end{conjecture}

Despite many partial results, only the case $\Delta=3$ has been proven~\cite{thomassen2001applications}. The conjecture is also known to be asymptotically true (and tight) for large values of $\Delta$~\cite{havet2017list,amini2013unified}. For a up-to-date overview of the best values depending on $\Delta$, we refer the reader to~\cite{bousquet2021improved}.

In this paper, we actually focus on the case $\Delta=4$. In that case, a greedy procedure shows that $17$ colors always suffice. Brooks' theorem brings down this bound to $16$ unless we consider a graph with $17$ vertices and diameter $2$, but this cannot hold for a planar graph. This result can again be improved to $15$ by a result from~\cite{cranston2016painting}. The bound was then brought down successively to $14$ and $13$ in~\cite{cranston2014choosability,zhu2021wegner}. We pursue this line of work by showing the following:

\begin{theorem}\label{thm:dist2}
Every planar graph $G$ of maximum degree $4$ satisfies $\chi_2(G) \le 12$.
\end{theorem}

\section{Preliminaries}
Given a graph $G$, we denote by $V(G)$ and $E(G)$ its vertex set and edge set, respectively. Note that the graphs we consider do not contain any loop or parallel edges. A graph $G$ is planar if its vertices and edges can be embedded in the plane in such a way no edges cross except maybe on vertices when they have this vertex as a common endpoint. When such an embedding is fixed, we denote by $F(G)$ the set of faces of $G$.

Let $v\in V(G)$. We denote by $N_G(v)$ the neighborhood of $v$ in $G$, \emph{i.e.} the set of vertices adjacent to $v$. The degree $\deg_G(v)$ of $v$ is $|N_G(v)|$. We extend this notion to faces by defining $\deg_G(f)$ as the number of edges incident to $f$, counted with multiplicity.

For an integer $d$, a vertex is said to be a \emph{$d$-vertex} (respectively, a \emph{$d^-$-vertex}, a \emph{$d^+$-vertex}) if its degree is equal to $d$ (resp., at most $d$, at least $d$). 
Likewise, a face is said to be a \emph{$d$-face} (resp., a \emph{$d^-$-face}, a \emph{$d^+$-face}) if its degree is equal to $d$ (resp., at most $d$, at least $d$).

Let $X \subseteq V(G)$. We denote by $G[X]$ the subgraph induced by the set of vertices $X$.
The graph obtained from $G$ by removing a vertex $v$ and its incident edges is denoted by $G-v$.
Note that $G - v = G [ V \setminus v]$.
Let $u, v$ in $V$.
By $G + uv$, we denote the graph $G'$ such that $V(G') = V(G)$ and $E(G') = E(G) \cup \{uv\}$, i.e., the graph obtained from $G$ by adding an edge between $u$ and $v$ if it does not already exist.

We end this section with coloring-related definitions. A \emph{partial distance-2 coloring} on a set of vertices $X$ of a graph $G$ is a coloring of $G[X]$ such that every pair of vertices at distance $2$ in $G$ are colored differently. In particular, when $X=V(G)$, a partial coloring is a coloring.

For each $v\in V(G)$, let $L(v)$ be a finite list of colors. We say that $G$ is \emph{$L$-choosable} if one can find a coloring of $G$ where each vertex $v$ receives a color from $L(v)$.  Note that when all the lists are equal, we recover the standard notion of coloring.  Let $\ell: V(G) \mapsto \mathbb{N}$ be a function. We say that $G$ is \emph{$\ell$-choosable} if, for any assignment of lists of colors $v\mapsto L(v)$ with $|L(v)|=\ell(v)$, $G$ is $L$-choosable.

\section{Overview of the proof}
\label{sec:overview}
The discharging method is a widely used proof technique, especially suited for obtaining results of the form ``\emph{every graph from the class $\mathcal{F}$ satisfies some property $\Pi$}'' (see~\cite{cranston2017introduction} for a survey on the discharging method). Usually, $\mathcal{F}$ is a class of sparse graphs. In our case, $\mathcal{F}$ is the class of planar graphs of maximum degree $4$, and $\Pi$ is the property $\chi_2(G)\leqslant 12$. 

We first consider the structure of the graphs in $\mathcal{F}$ to show that $G$ must contain some patterns (usually partially induced subgraphs with degree conditions) called \emph{configurations} (we say that the set of configurations is \emph{unavoidable}). To this end, 
we use a standard method called discharging. It consists in double counting some suitable quantity defined on a graph $G\in\mathcal{F}$. More precisely, we give some charges to some elements of $G$ (usually vertices and faces) such that the total charge, i.e., the sum of all the charges, is negative. We then redistribute the charges in $G$ to other vertices and faces (while preserving the total charge) with some \emph{discharging rules}. 
We then show that if some graph $G$ does not contain any configuration from our set of configurations, the final charge of each vertex and face of $G$ is non-negative. Since the initial total charge is negative and we keep charge when discharging, we obtain that our set of configurations is unavoidable. 

We then have to show that our set of unvoidable configurations are \emph{reducible}. Informally, it means that, for every graph $G\in\mathcal{F}$ containing one of them, one can construct a smaller graph $H$ such that if $H$ has a valid coloring, then so does $G$. Proving that all our configurations are reducible ends the proof. Indeed, consider a smallest counterexample to the sought result, \emph{i.e.} a graph $G\in\mathcal{F}$ such that $\chi_2(G)>12$ with minimum number of edges and vertices. It must then contain one of our configurations, which is reducible. This means that we can find a graph $H$ smaller than $G$, which would still be a counterexample (otherwise there would be a valid coloring for $H$ and we could extend it to $G$). This contradicts the minimality of $G$ and shows that $G$ cannot exist. 

We propose two algorithms that will automatize some parts of both steps. For the discharging part, our algorithm, which is the main contribution of this paper, aims at finding an unavoidable set of configurations.
More precisely, it tries to find some discharging rules using a Linear Program we will outline in Section~\ref{sec:lp}. Either our algorithm proves that such rules exist or will provide a list of configurations such that at least one of them has to be reduced to be able to reach a contradiction.

As a counterpart to this program, we need to test whether the obtained configurations are reducible. We present in Section~\ref{sec:heur} an algorithm that permits to reduce configurations. This algorithm is an heuristic whose goal consists in trying to prove that a graph is list-colorable. Our heuristic permits to keep a good balance between efficiency and quality (finding an efficient algorithm is hopeless since the problem is NP-complete).

Note that the two programs are not run sequentially: we actually go back and forth between them until we reach some kind of stabilization. Essentially, when our first algorithm outputs a list of candidate of unavoidable configurations, our second algorithm tries to reduce at least one of them. We repeat this process until we converge (see Section~\ref{sec:finalalgo} for more details).

In the following, we denote by $G$ a graph from $\mathcal{F}$ with $\chi_2(G)>12$ which minimizes $(|V(G)|,|E(G)|)$, and we fix an arbitrary embedding of $G$. We first reduce some technical configurations to ensure that our process and upcoming proofs are sound. This includes connectivity issues and vertices of degree at most $2$. In particular, we will assume in the following sections that all our graphs are $2$-connected and have minimum degree $3$.

\begin{lemma}
The graph $G$ is 2-connected.
\end{lemma}

\begin{proof}
Note that $G$ is connected, otherwise, some connected component of $G$ would be a smaller counterexample than $G$. Assume now that $G$ contains a cut vertex $v$, and denote by $C$ be the vertex set of a component of $G-v$. Note that $G_1 = G - C\in\mathcal{F}$ and $G_2 = G[C \cup \{v \}] \in\mathcal{F}$. Assume that they admit distance-$2$ $12$-colorings, respectively $\alpha_1$ and $\alpha_2$.

We can permute the colors of $\alpha_1$ in such a way $\alpha_1(v) = \alpha_2(v)$ and $\alpha_1(N_{G_1}(v)) \cap \alpha_2(N_{G_2}(v)) = \emptyset$. Now the union of $\alpha_1$ and $\alpha_2$ is a distance-$2$ $12$-coloring of $G$, which completes the proof.
\end{proof}

The main consequence of this lemma is that for every integer $d$, the vertices incident to a $d$-face induce a cycle. In that case, we denote the face by $[v_1,\dots v_d]$. 

\begin{lemma}
\label{lem:deg2}
The graph $G$ has minimum degree 3.
\end{lemma}

\begin{proof}
We only prove the result for 2-vertices, the other cases being straightforward. Assume that $G$ contains a 2-vertex. Let $G'$ be the graph obtained from $G$ by adding an edge between the neighbors of $v$ (if not already present) and removing $v$. 

Note that $G'\in\mathcal{F}$. By minimality, $G'$ admits a distance-$2$ $12$-coloring $\alpha$. Then $\alpha$ is a distance-$2$ coloring of $G-v$, and at most $8$ colors are present within distance $2$ from $v$. Therefore, there always exists an available color to extend $\alpha$ to $v$. 
\end{proof}

\section{Unavoidable set of configurations}
\label{sec:lp}
\subsection{Encoding rules as a Linear Program}

To find the discharging rules and the corresponding set of unavoidable configurations, we implement an algorithm based on linear programming. This algorithm, given a set of configurations $\mathcal{C}$, returns, if it exists, a set of discharging rules that ensures that the final charge of every face and vertex not appearing in a configuration in $\mathcal{C}$ is non-negative. If this set does not exist, then the set of forbidden configurations must be updated and the algorithm will give us new configurations, some of them having to be reduced and added to our set $\mathcal{C}$.

We encode some discharging rules using the variables of a Linear Program to represent the amount of charge displaced. We encode the fact that the final charges are non-negative in the constraints of the program. There exists infinitely many ways to transfer rules from vertices/faces to other vertices/faces. In this paper, we focus on so-called local discharging rules where charge is transferred locally to close vertices and faces. In order to avoid combinatorial blow up, we will actually assume that we transfer charge only to adjacent vertices or faces. One can indeed consider more general rules but have to pay a price in terms of running time.
Note that most of the statements using the discharging method only use this type of rules. Also note that it does not reduce the problem into a finite problem since there might exist infinitely many types of neighborhoods, and the charges are real numbers.

\paragraph{Initial charges.}
Our initial charge distribution is given as follows: each vertex and face $x$ receives $\deg(x)-\alpha$, where $\alpha$ is a variable of our linear program we want to maximize. The constraints of the Linear Program will depend on a set of configuration $\mathcal{C}$ that will be formally defined later. Informally speaking, the set $\mathcal{C}$ will correspond to the list of reducible configurations which cannot appear in a minimal counter-example. So, for each configuration of $\mathcal{C}$, we will be able to remove some constraints in the Linear Program.

\begin{theorem}
\label{thm:LP}
Let $P$ be the linear program created from a set $\mathcal{C}$ of configurations, and $\alpha^*$ be the optimal value of $\alpha$ in $P$. Then there exists a set of discharging rules such that, when applied to any graph $G\in\mathcal{F}$ avoiding $\mathcal{C}$, all vertices and faces $x$ end with non-negative weight if their initial weight was $\deg_G(x)-\alpha^*$.
\end{theorem}

One may consider $\alpha$ as an informal measurement of how far we are from obtaining a proof. For planar graphs, if we find a set $\mathcal{C}$ of reducible configurations such that $\alpha\geqslant 4$, we know that the total charge is negative because of the following equality (derived from Euler's formula). 

\begin{equation}
 \sum_{x \in V(G)\cup F(G)}(\deg(x) - 4) = - 8
 \label{eq:Euler}
\end{equation}

Assuming that Theorem~\ref{thm:LP} holds, we reach a contradiction by applying it to our minimal counterexample $G$.

The end of this section is thus devoted to defining the Linear Program so that Theorem~\ref{thm:LP} holds.

\paragraph{Variables of the linear program.}
We need some definitions on the adjacencies of faces and vertices. An edge $uv$ is \emph{incident} to a face $f$ if $uv$ is an edge of the border of the face. A face $f$ is \emph{edge-adjacent} to a face $f'$ if there is an edge $uv$ called \emph{common edge} such that $uv$ is incident to $f$ and $f'$ (see Figure~\ref{fig:rules} (c)).
A face $f$ is \emph{vertex-adjacent} to a face $f'$ if there is a unique vertex $v$ called \emph{common vertex} such that $v$ is incident to $f$ and $f'$ (see Figure~\ref{fig:rules} (b)).
A face $f$ is \emph{adjacent} to a face $f'$ if $f$ is vertex-adjacent or edge-adjacent to $f'$.
A face $f$ is \emph{adjacent} to a vertex $v$ if $v$ is incident to $f$.

We will allow the following three types of discharging rules, where $d\geqslant 5$, $(\ell,r,\ell',r') \in \mathbb{N}_{\geqslant 3}^4$: 
\begin{itemize}
    \item A $d$-face $x$ transfers to each incident $3$-vertex $y$ a charge $\omega_{d, \ell, 3v,r}$ if $y$ is also incident to an $\ell$-face and a $r$-face (see Figure~\ref{fig:rules}(a)). 
    \item A $d$-face $x$ transfers to each vertex-adjacent $3$-face $y$ a charge $\omega_{d, \ell, 3,r}$ if there are a $\ell$-face edge-adjacent to $x$ and $y$, and also a $r$-face edge-adjacent to $x$ and $y$  (see Figure~\ref{fig:rules}(b)).
    \item A $d$-face $x$ transfers to each edge-adjacent $3$-face a charge $\omega_{d,\ell,\ell',3,r',r}$ if $x$ and $y$ have the common $4$-vertices $u$ and $v$ such that, $u$ is incident to a $\ell$-face edge-adjacent to $x$ and a $\ell'$-face vertex-adjacent to $x$, and $v$ is incident to a $r$-face edge-adjacent to $x$ and a $r'$-face vertex-adjacent to $x$ (see Figure~\ref{fig:rules}(c)).
\end{itemize}
 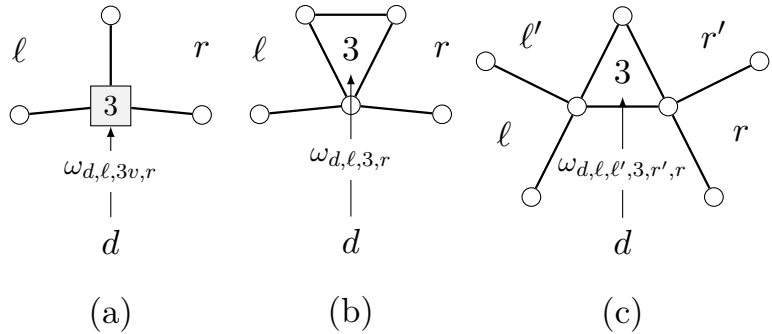
\begin{figure}[hbtp]
        \begin{center}
        \tikzstyle{vertexa}=[circle,draw, minimum size=15pt, scale=1, inner sep=1pt, fill=black!5, very thick]
        \tikzstyle{vertexbc}=[circle,draw, minimum size=15pt, scale=1, inner sep=1pt, fill=black!5]
        \tikzstyle{3vertex}=[draw, minimum size=15pt, scale=1, inner sep=1pt, fill=black!5]
        \tikzstyle{plain}=[circle,draw=none, minimum size=15pt, scale=1.2, inner sep=1pt]
        \tikzstyle{vertex}=[circle,draw, minimum size=7pt, scale=1, inner sep=1pt]
        \tikzstyle{fleche}=[->,>=latex]
        
        \begin{tikzpicture}[scale=1.2]
            \node (b) at (0,-0.1) [vertex] {};
            \node (c) at (1,0) [3vertex] {$3$};
            \node (d) at (2, -0.1) [vertex] {};
            \node (x) at (1, 1) [vertex] {};
            \node (y) at (1, -1.5) [plain] {$d$};
            \node (a1) at (0, 0.62) [plain] {$\ell$};
            \node (a2) at (2, 0.62) [plain] {$r$};
            
            \draw [fleche](y) to  node[pos=0.5,fill=white] {$\omega_{d,\ell,3v,r}$} (c); 
            \draw [-, line width=0.03cm] (b) to (c);
            \draw [-, line width=0.03cm] (c) to (d);
            \draw [-, line width=0.03cm] (x) to (c);

            \node (cap) at (1, -2.3) [plain] {(a)};
        \end{tikzpicture}
            \begin{tikzpicture}[scale=1.2]
            \node (b) at (0,-0.1) [vertex] {};
            \node (c) at (1,0) [vertex] {};
            \node (d) at (2, -0.1) [vertex] {};
            \node (x) at (1.5, 1) [vertex] {};
            \node (z) at (0.5, 1) [vertex] {};
            \node (y) at (1, -1.5) [plain] {$d$};
            \node (a) at (1, 0.62) [plain] {$3$};
            
            \node (a1) at (0, 0.62) [plain] {$\ell$};
            \node (a2) at (2, 0.62) [plain] {$r$};
            
            \draw [fleche](y) to  node[pos=0.4,fill=white] {$\omega_{d,\ell,3,r}$} (a); 
            \draw [-, line width=0.03cm] (b) to (c);
            \draw [-, line width=0.03cm] (c) to (d);
            \draw [-, line width=0.03cm] (x) to (c);
            \draw [-, line width=0.03cm] (z) to (c);
            \draw [-, line width=0.03cm] (z) to (x);

            \node (cap) at (1, -2.3) [plain] {(b)};
        \end{tikzpicture}
       \begin{tikzpicture}[scale=1.2]
            \node (b) at (0,0) [vertex] {};
            \node (c) at (1,0) [vertex] {};
            \node (d) at (1.5, -1) [vertex] {};
            \node (d1) at (-0.5, -1) [vertex] {};
            \node (i) at (0.5, 1) [vertex] {};
            \node (x) at (0.5, 0.38) [plain] {$3$};
            \node (y) at (0.5, -1.5) [plain] {$d$};
            \node (k) at (2, 0.5) [vertex] {};
            \node (k1) at (-1, 0.5) [vertex] {};
            
            \node (a1) at (-0.5, 0.8) [plain] {$\ell'$};
            \node (a2) at (1.5, 0.8) [plain] {$r'$};
            \node (b1) at (-0.8, -0.3) [plain] {$\ell$};
            \node (b2) at (1.8, -0.3) [plain] {$r$};

            \draw [fleche](y) to  node[pos=0.4,fill=white] {$\omega_{d,\ell, \ell',3, r', r}$} (x); 
            \draw [-, line width=0.03cm] (b) to (c);
            \draw [-, line width=0.03cm] (c) to (d);
            \draw [-, line width=0.03cm] (c) to (i);
            \draw [-, line width=0.03cm] (b) to (i);
            \draw [-, line width=0.03cm] (c) to (k);
            \draw [-, line width=0.03cm] (b) to (k1);
            \draw [-, line width=0.03cm] (b) to (d1);

            \node (cap) at (0.5, -2.3) [plain] {(c)};
        \end{tikzpicture}

        \end{center}
        \caption{The three types of discharging rules.}
        \label{fig:rules}
        \end{figure}

We also define the opposite rules that tell how much charge a $3$-face or a $3$-vertex receives according to the previous rules.
Thus, we define the charges $\omega_{3v, r, d, \ell} = - \omega_{d,\ell,3v,r}$ that a $3$-vertex receives from an incident $d$-face, the charges $\omega_{3, r, d, \ell} = - \omega_{d,\ell,3,r}$ that a $3$-face receives from a vertex-adjacent $d$-face and the charges $\omega_{3, r',r,d,\ell,\ell'} = - \omega_{d, \ell,\ell',3,r',r}$ that a $3$-face receives from an edge-adjacent $d$-face.
Note that, by the definition, the functions $(\omega_i)_{i \in I}$ are symmetric in a sense that: $\omega_{d, \ell, d',r} = \omega_{d, r, d',\ell}$ and $\omega_{d, \ell, \ell', d', r', r} = \omega_{d, r, r', d',\ell',\ell}$.

To each of these rules $(\omega_i)_{i \in I}$, we associate a variable of our Linear Program, so that a solution to the Linear Program will assign to each rule the amount of charge it transfers.

\paragraph{Constraints of the Linear Program}
Now we have the variables of the Linear Program, let us explain what are the constraints. We basically want that, at the end of the procedure, all the vertices and faces have a non negative charge.

In order to understand how the charge is modified around a vertex, we need to know the configuration of the faces and vertices around it. 
A \emph{configuration around} a $d$-face $f = [f_1,\dots, f_d]$ (see Figure~\ref{fig:exconfig}) is a cyclic ordering $[d_i]_{i\le 2d} \in (\{ 3v\}\cup\mathbb{N}_{\geqslant 3})^{2d}$ such that, for every $1 \le i \le d$:
\begin{itemize}
    \item if $v_i$ is a $3$-vertex, $d_{2i + 1} = 3v$. 
    \item if $v_i$ is a $4$-vertex, the face vertex-adjacent to $f$ with the common vertex $v_i$ is a $d_{2i + 1}$-face.
    \item the face edge-adjacent to $f$ with the common edge $v_iv_{i+1}$ is a $d_{2i}$-face.
\end{itemize}

\begin{figure}[hbtp]
        \begin{center}
        \tikzstyle{vertex}=[circle, draw, minimum size=7pt, scale=1, inner sep=1pt]
        \tikzstyle{vertexa}=[circle, draw, minimum size=7pt, scale=1, inner sep=1pt, thick]
        \tikzstyle{vertexb}=[circle, draw = none, minimum size=7pt, scale=0.9, inner sep=1pt, thick]
        \tikzstyle{plain}=[circle,draw=none, minimum size=15pt, scale=1.2, inner sep=1pt]
        \tikzstyle{fleche}=[->,>=latex]
        
      \begin{tikzpicture}[scale=1.1]
            \node (a) at (-0.30901699437, -0.95105651629) [vertexa] {$v_1$};
            \node (b) at (0,0) [vertexa] {$v_2$};
            \node (c) at (1,0) [vertexa] {$v_3$};
            \node (d) at (1.30901699437, -.95105651629) [vertexa] {$v_4$};
            \node (e) at (0.5, -1.39680224667) [vertexa] {$v_5$};
            \node (f) at (-1.17504239816,  -0.45105651629) [vertex] {};
            \node (g) at (-0.86602540378, 0.5) [vertex] {};
            \node (h) at (-0.36602540378, 1.366025404) [vertex] {};
            \node (i) at (0.5, 0.866025404) [vertex] {};
            \node (j) at (1.36602540378, 1.366025404) [vertex] {};
             \node (k) at (1.86602540378, 0.5) [vertex] {};
            \node (l) at (2.17504239816,  -0.45105651629) [vertex] {};
            
            \node (q) at (-1.1, -1.5) [vertex] {};
            \node (m) at (1.1, -2) [vertex] {};
            \node (n) at (2.1, -1.5) [vertex] {};
            \node (r) at (0.5, -2.4) [vertex] {};
            \node (s) at (-0.3, -2.5) [vertex] {};
            \node (t) at (-0.8, -2) [vertex] {};
           
            \draw [line width=0.03cm] (a) to (b);
            \draw [line width=0.03cm] (a) to (e);
            \draw [line width=0.03cm] (b) to (c);
            \draw [line width=0.03cm] (c) to (d);
            \draw [line width=0.03cm] (d) to (e);
            \draw  (c) to (i);
            \draw  (b) to (i);
            \draw  (f) to (g);
            \draw  (g) to (h);
            \draw  (j) to (k);
            \draw  (k) to (l);
            \draw  (c) to (k);
            \draw  (h) to (i);
            \draw  (f) to (a);
            \draw  (d) to (l);
            \draw  (i) to (j);
            \draw  (a) to (q);
            \draw  (f) to (q);
            \draw  (e) to (m);
            \draw  (m) to (n);
            \draw  (n) to (l);
            \draw  (r) to (s);
            \draw  (s) to (t);
            \draw  (m) to (r);
            \draw  (q) to (t);
            
            \draw[fleche, very thick] (2.5,-0.5) to (3.5, -0.5);
            \node (y) at (6.5, -0.5) [plain] {$[3,6^+,3v,3,4,4,3v,5,3v, 6^+]$};
            
        \end{tikzpicture}
        \end{center}
        \caption{Example of a configuration around a $5$-face}
        \label{fig:exconfig}
    \end{figure}
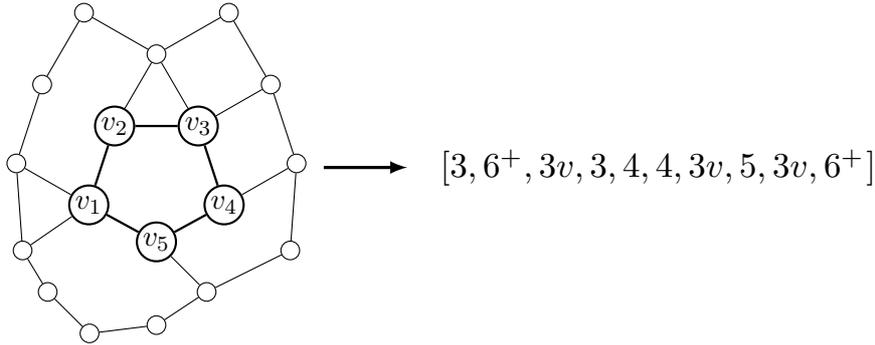

A \emph{configuration around} a $3$-vertex $v$ is the triple of integers corresponding to the degrees of the faces incident to $v$.

Let us now explain what will be the constraints of our Linear Program.
The goal of the proof is to transfer charges in order to get a non-negative final charge for every face and every vertex outside of configurations in $\mathcal{C}$.
Thus, the charge transferred $(\omega_i)_{i \in I}$ must satisfy the following constraints.

Let $f$ be a face of degree $d$. 
For each configuration $[d_i]_{i\le2d}$ around $f$, we can construct the set of variables $I([d_i]_{i \le 2d}) \subseteq I$  that will be applied on $f$.
In order to have a non-negative final charge for $f$ given an initial charge $d - \alpha$, the charges transferred $(w_i)_{i \in I}$ must verify that, for every possible configuration, we have:
$$ (d - \alpha) - \sum_{i \in I([d_i]_{i \le 2d})}\omega_i \ge 0 \quad (\star)$$ 
Likewise, the constraint on a configuration around a $3$-vertex is of the form:
$$ (3 - \alpha) - \sum_{i \in I([d_i]_{i\le 3})}\omega_i \ge 0 \quad (\star\star)$$ 
Note that the constraints $(\star)$ and $(\star\star)$ only depend to the configurations represented by a cycle of the form $[d_i]$ and thus the constraints are well-defined. Observe also that there is no constraint associated with vertices of degree at most 2, or 4. The former are reducible by Lemma~\ref{lem:deg2}. The latter are actually not affected by our rules: no charge is taken nor given to them (so that when $\alpha=4$, their initial and final charge is 0).

\subsection{Trimming the Linear Program}

The current encoding ensures that Theorem~\ref{thm:LP} holds. However, in its current state, our linear program contains infinitely many variables and configurations, which is obviously problematic. We will therefore trim it to obtain a finite program that will be enough to still get a weaker version of Theorem~\ref{thm:LP} (by adding some hypothesis on the set $\mathcal{C}$ of configurations). At the end of this operation, we get a finite linear program whose solutions are solutions of the initial one. However, the converse is not necessarily true: some proof power is lost, and the harsher the trimming, the weaker our system will be. 
So, for the trimming, we have to find the balance between relevant information to get a proof and number of variables and constraints (and then running time). In what follows, we explain how we can trim the LP to get a proof of Theorem~\ref{thm:dist2}, it might be slightly different if we want to prove other statements.

The trimming is twofold. First, we reduce the number of variables. We do so by not considering the exact length of faces as long as they are larger than some threshold. Increasing the threshold increases the proof power of the technique, to the price of the running time. In our case, we fix the threshold to six, meaning that we do not distinguish hexagons from heptagons, octagons, etc. More precisely, when considering the variables $\omega_{d,\ell,3v,r}, \omega_{d,\ell,3,r},\omega_{d,\ell,\ell',3,r',r}$, we now assume that $d \in \{5, 6^+\}$ and $(\ell,r,\ell',r') \in \{3,4,5,6^+\}^4$, so that our linear program now has finitely many variables.

We then reduce the number of constraints. Since the rules we consider are local (in the sense they only consider incident faces and vertices of a face or a $3$-vertex), note that there are at most $4^3$ configurations to consider for a $3$-vertex and, for each $d\geqslant 3$, at most $5^{2d}$ of them for a $d$-face. Therefore, the only reason for which we get infinitely many constraints is that $d$ can take any value greater than $2$. 

To solve this, we end this section by proving that only six properties imply that the final charge of $6^+$-faces is non-negative, provided we reduce some finite set $\mathcal{C}_{6^+}$ of configurations. This allows us to replace all the constraints for faces of size at least $6$ by setting the value of some variables of our Linear Program and reducing some set $\mathcal{C}_{6^+}$ of configurations. 

We define the six discharging rules \textbf{T} on $6^+$-faces as follows:
\begin{itemize}
    \item[\textbf{T1}] A $6^+$-face gives charge $\frac{2}{3}$ to each incident $3$-vertex.
    \item[\textbf{T2}] A $6^+$-face $f$ gives charge  $\frac{2}{3}$ to each edge-adjacent $3$-face $x$ if there is a $3$-face vertex-adjacent to $f$ and edge-adjacent to $x$.
    \item[\textbf{T3}] A $6^+$-face $f$ gives charge $\frac{2}{3}$ to each edge-adjacent $3$-face $x$ if $f$ and $x$ have a common vertex incident to two $4$-faces.
    \item[\textbf{T4}] A $6^+$-face $f$ gives charge $\frac{1}{3}$ to each edge-adjacent $3$-face $x$ if we are not in case \textbf{T2} or \textbf{T3}.
    \item[\textbf{T5}] A $6^+$-face $f$ gives charge $\frac{1}{3}$ to each vertex-adjacent $3$-face $x$ if there is a $4$-face edge-adjacent to $f$ and $x$.
    \item[\textbf{T6}] A $6^+$-face $f$ gives $0$ to each vertex-adjacent $3$-face $x$ such that \textbf{T5} does not apply. 
\end{itemize}

In other words, we fix some variables in our Linear Program, namely if $(\ell,r,\ell',r') \in \{3,4,5,6^+\}^4$, we set:
\begin{itemize}
    \item[\textbf{T1}] $\omega_{6^+,\ell, 3v,r}  = \frac{2}{3}$.
    \item[\textbf{T2}] $\omega_{6^+,\ell', \ell, 3,3, r'} = \omega_{6^+, \ell',3, 3, 3,r'} = \omega_{6^+, \ell',3, 3, r,r'} = \frac{2}{3}$.
    \item[\textbf{T3}]$\omega_{6^+,\ell',\ell, 3, 4,4} = \omega_{6^+, 4, 4, 3, r,r'} =  \frac{2}{3}$.
    \item[\textbf{T4}] If $\ell \neq 3$, $r \neq 3$, $(\ell',\ell) \neq (4,4)$ and $(r,r') \neq (4,4)$, then $\omega_{6^+,\ell', \ell, 3, r, r'} = \frac13$.
    \item[\textbf{T5}]$\omega_{6^+,\ell, 3,4} = \omega_{6^+, 4,3,r} =\frac13$.
    \item[\textbf{T6}]If $\ell\neq 4$ and $r\neq 4$, then $\omega_{6^+,\ell, 3,r } =0$.
\end{itemize}

We now prove that the main result of this section, namely that our trimming is correct.

\begin{lemma}
\label{lem:6+}
Assume the charges transferred $(\omega_i)_{i \in I}$ satisfy \textbf{T}. There is a finite set of configurations $\mathcal{C}_{6^+}$ such that for every graph not containing them, the final charge of every $6^+$-face is non-negative.
\end{lemma}

Assuming that Lemma~\ref{lem:6+} holds, we directly obtain Theorem~\ref{thm:LP} when we consider a set of configurations containing $\mathcal{C}_{6^+}$. Indeed, our trimmed Linear Program satisfies \textbf{T}, which ensures that the rules we obtain yield non-negative weights on $6^+$-faces. Moreover, the constraints corresponding to $5^-$-faces and $3$-vertices stayed unchanged, hence these faces and vertices also end with non-negative weight. We end this section with the proof of Lemma~\ref{lem:6+}.

\begin{proof}
Let $f$ be a face of degree $d \ge 6$, with an initial charge of $d-4$. We will analyse how much charge is lost by $f$ when applying rules \textbf{T1}-\textbf{T6}. To this end, we will make the charge go through some intermediate stop before reaching its final destination. More precisely, we will restate the rules so that the charge goes through some edges of the boundary of $f$. Thus, $f$ gives the charges to edges and then the edges will give the charges to the $3$-faces and $3$-vertices. This equivalent reformulation, quite standard in discharging proofs, permits to simplify the analysis of the charge leaving $f$.

Consider the different cases where the rules with positive charge are applied (see Figure~\ref{fig:larger}):
\begin{itemize}
    \item if $f$ gives $\frac23$ to an incident $3$-vertex $v$ by \textbf{T1}, we say that $f$ gives $\frac{1}{3}$ to the two edges incident with $v$ which are on the boundary of $f$. 
    \item if $f$ gives charge to an edge-incident $3$-face $x$ by \textbf{T2}-\textbf{T4}, we first give $\frac13$ to their common edge $e$. If moreover \textbf{T2} or \textbf{T3} applies, we also split the remaining $\frac13$ between the edges of $f$ incident to $e$ if their common endpoint satisfies the condition of the rule.
    \item if $f$ gives charge to a vertex-incident $3$-face by \textbf{T5}, let $v$ be their common vertex and let $y$ be the 4-face edge-incident to $x$ and $f$. If $f$ is vertex-incident to another 3-face sharing an edge with $y$, then $f$ gives $\frac16$ to each edge incident to $f$ and $v$. Otherwise, it gives $\frac 13$ to its common edge with $y$.
\end{itemize}
\input{remaining_configurations/largerfaces}

Let $\mathcal{C}_{6^+}$ be the set of all structures leading to more than $\frac13$ transiting through an edge. Note that the amount of charge transiting through an edge $uv$ depends only on the configuration of $4^-$-faces incident to $u$ and $v$. Therefore, $\mathcal{C}_{6^+}$ contains finitely many configurations.

Now, if no such configuration arises in $G$, then at most $\frac13$ goes through each of the $d$ edges. Therefore, the final charge of the face $f$ is $d-4-\frac{d}{3}$, which is non-negative when $d\geqslant 6$. This concludes the proof. 
\end{proof}

We postpone the exact description of configurations in $\mathcal{C}_{6^+}$ to Section~\ref{sec:C6+}, where we show their reducibility (either manually or using our heuristic). 

\subsection{The final algorithm}
\label{sec:finalalgo}
We now have all the ingredients to present our program.
Let $D$ be the set of all configurations around faces and $3$-vertices represented by the cycles $[d_i]$ and let $\mathcal{C}$ be a set of forbidden configurations.
For every configuration $[d_i]$ of $D$, we have a constraint of the form ($\star$) or ($\star\star$). 
We solve the following linear program:
$$ \max_{\alpha, \omega}{\alpha} \quad\text{subject to}$$ 
$$ \ (d - \alpha) - \sum_{i \in I([d_i]_{i \le 2d})}\omega_i \ge 0 \ \ \text{for} \ d \ge 3 \ \text{and} \ [d_i]_{i \le 2d} \in D\setminus \mathcal{C}$$
$$   (3 - \alpha) - \sum_{i \in I([d_i]_{i \le 3})}\omega_i \ge 0 \\ \ \text{for} \ [d_i]_{i \le 3} \in D\setminus \mathcal{C}$$
The linear program returns the maximum $\alpha$ and a set of charges $(\omega_i)_{i \in I}$ such that the final charge of every face and vertex not in a configuration from $\mathcal{C}$ is non-negative after applying the rules with the charges $(\omega_i)_{i \in I}$.

If the maximum $\alpha$ is greater or equal to $4$, then the constraints are in particular verified for $\alpha = 4$. This means that each vertex and face not appearing in a configuration from $\mathcal{C}$ ends with non-negative charge, even if the initial charges are $\deg(x) - 4$ and the total charge is negative by Equation~(\ref{eq:Euler}). We thus obtain that $\mathcal{C}$ is unavoidable, which ends the proof, provided that each configuration of $\mathcal{C}$ is reducible. 

At the beginning, there is no chance that the linear program will output some $\alpha \ge 4$ since we have not set any constraint on our planar graph. So the maximum $\alpha$ given by our solver will be less than $4$. So our Linear Program will output the maximum value of $\alpha$ we can obtain and for which extreme point of the polyhedron of the feasible region we can obtain it (an extreme point of a Linear Program being defined as the intersection of the constraints tight on this point). In particular, it ensures that if we want to improve the solution one of these constraints has to be removed.

Our algorithm returns this set of tight constraints, each of them corresponding to some configurations in $D$. Now, our goal is to prove that some of these configurations are reducible. If this is possible, we can remove the corresponding constraints from the linear program (the configurations are added in the set of reducible configurations $\mathcal{C}$), which (hopefully) allows the maximum value of $\alpha$ to increase.

We now repeat the loop again and again until either we end up with $\alpha \ge 4$ (and hence the theorem is proven), or we fail to prove that some configuration is reducible (which means the system has not enough expression power to prove the theorem, and a less harsh trimming has to be considered). To sum up, we represent on Figure~\ref{fig:prog} the behavior of our program. 

 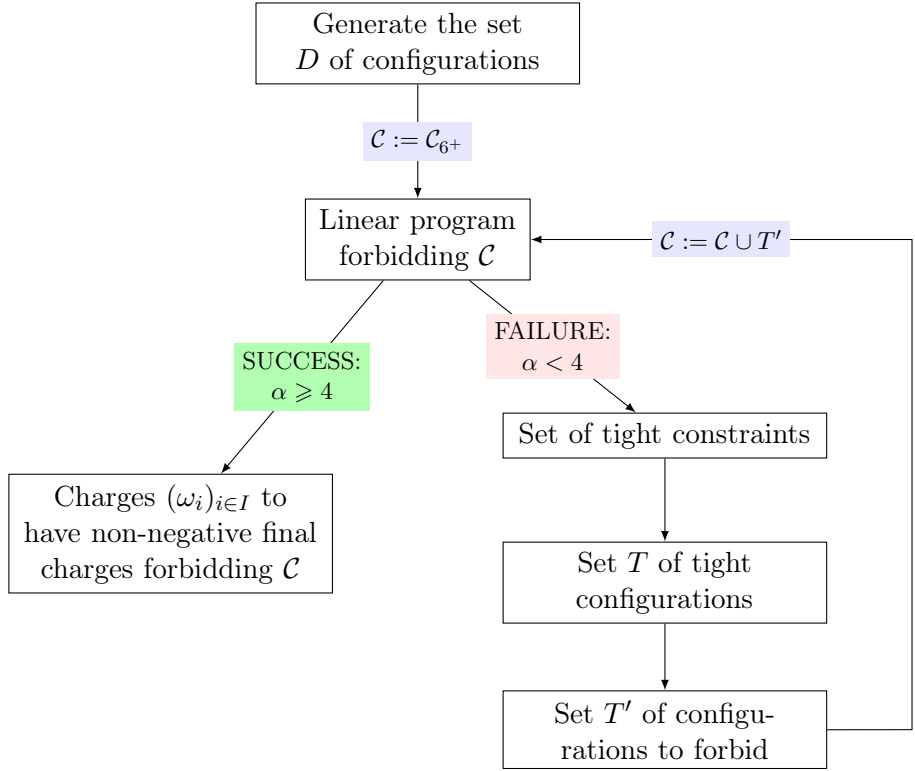
\begin{figure}[hbtp]
        \begin{center}
        \tikzstyle{vertexa}=[circle,draw, minimum size=15pt, scale=1, inner sep=1pt, fill=black!5, very thick]
        \tikzstyle{vertexbc}=[circle,draw, minimum size=15pt, scale=1, inner sep=1pt, fill=black!5]
        \tikzstyle{3vertex}=[draw, minimum size=15pt, scale=1, inner sep=1pt, fill=black!5]
        \tikzstyle{plain}=[circle,draw=none, minimum size=15pt, scale=1.2, inner sep=1pt]
        \tikzstyle{vertex}=[circle,draw, minimum size=7pt, scale=1, inner sep=1pt]
        \tikzstyle{fleche}=[->,>=latex]
        
        \begin{tikzpicture}[scale=1.3]
              \node (a) [draw,text width=2.7cm, text justified, align=center] at(0,0){ 
                Linear program forbidding $\mathcal{C}$   
             };
              \node (b) [draw,text width = 4cm, text justified, align = center] at(-2.5,-3){ 
                Charges $(\omega_{i})_{i \in I}$ to have non-negative final charges forbidding $\mathcal{C}$};
             
             \draw [fleche] (a) to node[pos=0.5,fill=green!30, align =center, scale = 0.85] 
             {
             SUCCESS: \\ $\alpha \ge 4$
             } (b);
             
               \node (x) [draw,text width = 4cm, text justified, align = center] at(0,2){ 
                Generate the set $D$ of configurations 
             };
            
              \node (c) [draw,text width = 4cm, text justified, align = center] at(2.5,-2){ 
                Set of tight constraints 
             };
             
             \draw [fleche] (a) to node[pos=0.5,fill=red!10, scale = 0.85, align = center] {FAILURE: \\ $\alpha < 4$} (c);

            \node (d) [draw,text width = 4cm, text justified, align = center] at(2.5,-3.5){ 
                Set $T$ of tight configurations 
             };

            \node (e) [draw,text width = 4cm, text justified, align = center] at(2.5,-5){ 
                Set $T'$ of configurations to forbid
             };
             \draw [fleche] (c) to (d);
             \draw [fleche] (d) to (e);
             \draw [fleche] (x) to node[pos=0.5,fill=blue!10, scale = 0.85, align = center] {$\mathcal{C} := \mathcal{C}_{6^+}$} (a);
             \draw [fleche] (e) to (5, -5) to (5,0) to node[pos=0.5,fill=blue!10, scale = 0.85] {$\mathcal{C} := \mathcal{C} \cup T'$} (a);
            
        \end{tikzpicture}

        \end{center}
        \caption{Construction of a proof using the discharging method.} 
        \label{fig:prog}
        \end{figure}

\paragraph{Implementation.}

Our system is implemented in Python and runs on a standard laptop. The code can be found at \url{https://gitlab.com/lucas.de_meyer/square-coloring-stage}.

As a pre-processing step, we generate the set $D$ of all configurations (filtering out symmetries) as well as the initial linear program. As shown before, we just need the configurations around the $3$-vertices, the $3$-faces and the $5$-faces, so configurations of length respectively $3$, $6$ and $10$.
Generating all these configurations takes nearly $20$ minutes and it can be done just once, since we then only remove some constraints.

Afterwards, at each iteration, the program removes the reducible configurations from $D$, generates the corresponding linear program and solves it, in less than $2$ minutes in total.

In our approach, we update manually the set $\mathcal{C}$ of reducible configurations.
Indeed, the linear program returns a set of tight configurations and, among them, we choose at least one of their subgraphs which we can reduce. Note that this choice is not random, we instead aim at removing as many tight constraints as possible from $D$.

After a dozen iterations (for less than an hour of computation in total), we are able to get a set $\mathcal{C}$ of 40 configurations such that the associated Linear Program has a solution where $\alpha=4$. A more precise description of this set is provided in the next section. By Lemma~\ref{lem:6+}, we obtain the following.

\begin{theorem}
\label{thm:unavoid}
The minimum counter-example $G$ contains a configuration from $\mathcal{C}\cup\mathcal{C}_{6^+}$.
\end{theorem}

\section{Reducing configurations}
\label{sec:heur}
The goal of this section is to conclude the proof of Theorem~\ref{thm:dist2} by reducing the configurations from $\mathcal{C}\cup\mathcal{C}_{6^+}$. We first present another program in Section~\ref{subsec:heur} that allows us to reduce $28$ out of the $40$ configurations in $\mathcal{C}$. We then reduce manually the $12$ remaining configurations from $\mathcal{C}$ in Section~\ref{sec:remaining}, using some auxiliary results proved in Section~\ref{sec:sepcyc}. Finally, we conclude the proof  in Section~\ref{sec:C6+} by showing that if $G$ contains a configuration from $\mathcal{C}_{6^+}$, it actually contains another configuration from $\mathcal{C}$. 

\subsection{A heuristic}
\label{subsec:heur}
\paragraph{How to reduce configurations.}
We start with a description of a generic method used to reduce a configuration $C$. We follow the lines of the proof of Lemma~\ref{lem:deg2}. Recall that our goal is, assuming that our counterexample $G$ contains $C$, to provide a graph $G_C$ with less vertices (or edges) than $G$ and such that one can extend a distance-$2$ coloring $\alpha$ of $G_C$ with $12$ colors (which exists by minimality) to a coloring of $G$. 
For our purposes, $G_C$ will most of the time be defined by removing from $G$ some vertices involved in $C$.

For each vertex $v \in V(C)$, the coloring $\alpha$ forbids some colors for $v$ and we can consider $L(v)$ the set of \emph{free colors} of $v$, that is the set of colors $c$ such that no vertex at distance at most $2$ from $v$ in $V(G_C)$ is colored $c$.
Each set $L(v)$ has a minimum possible size $\ell(v)$ which is the total number of colors minus the number of vertices of $V(G_C)$ at distance at most $2$ from $v$.
Let $H$ be the graph whose vertices are the uncolored vertices of $G$, with an edge between pairs of vertices within distance 2 in $G$. Let $\ell$ be the function that asigns to every vertex $v \in H$ the integer $\ell(v)$. Our goal is to show, that $H$ is \emph{$\ell$-choosable}. If we can prove it, it ensures that any possible coloring of $G_C$ (and in particular a distance-$2$ $12$-coloring) can be extended to $G$, which proves that $C$ is reducible.

Graph choosability is well-known to NP-hard. To improve performance, instead of solving this problem exactly, We implement an heuristic algorithm $\A$ that, given a graph $H$ and $\ell : V(H) \mapsto \mathbb{N}$, if $\A(H, f)$ returns \texttt{True}, then $H$ is $\ell$-choosable. 
And then, we apply $\A$ on each of the graphs $H$ obtained from the $41$ configurations given by our Linear Program algorithm.
If $\A$ returns \texttt{True}, we have proved that the configuration $C$ is reducible. Note however that a negative answer does not imply the non-choosability of the graph, nor the irreducibility of $C$.

In the following, we present some generic techniques that, once combined, allow to reduce $30$ of the $41$ configurations. It is worth mentioning that the method we provide here is generic for any type of coloring problem. For the ease of exposition, we thus state the method for standard (\emph{i.e.} distance-$1$) coloring (that is, we simply want that adjacent vertices receive distinct colors). This is not restrictive, indeed, recall that a distance-$2$ coloring of $H$ is a distance-$1$ coloring of $H^2$.

\paragraph{Happy vertices.}
Some vertices might be easy to color. We will get rid of them, to focus on the hardest part of the instance. Let $H$ be a graph and $\ell : V(H) \mapsto \mathbb{N}$.
A vertex $v$  is \textit{happy} in $H$ for $\ell$ if $\ell(v) > \deg(v)$. 
Let $v$ be a happy vertex in $H$ for $\ell$.
For any coloring of $H - v$, we can extend the coloring to $v$ for any list of length $\ell(v)$, since we forbid at most $\deg(v) < \ell(v)$ colors.
Thus, if $H - v$ is $\ell$-choosable, $H$ is $\ell$-choosable. Hence, instead of coloring $H$, it is sufficient to remove its happy vertices (without modifying the value of $\ell$ for the other vertices) and color the remaining graph. This yields the following.

\begin{lemma}\label{lem:happy}
Let $v$ be a happy vertex of a graph $H$ for a function $\ell : V(H) \mapsto \mathbb{N}$. If $H - v$ is $\ell$-choosable, then $H$ is $\ell$-choosable.
\end{lemma}

\paragraph{Coloring a vertex.}
Let $H$ be a graph, $L$ be a list assignment and $v\in V(H)$. Suppose that we want to color $v$ with a color $c\in L(v)$. For each neighbor $u$ of $v$ such that $c \in L(u)$, the number of free colors of $u$ decreases. 

Let $\ell_{v,c} : V(H) \mapsto \mathbb{N}$ be the resulting numbers of colors, \emph{i.e.}  $\ell_{v,c}(u) = |L(u) \setminus c|$ if $u$ is a neighbor of $v$, and $\ell_{v,c}(u) = |L(u)|$ otherwise. Since reducing the sizes of the lists only makes our task harder, we may actually assume that the list of every neighbor $u$ of $v$ has lost a color, so that $\ell_{v,c}(u) = \ell(u) - 1$. Therefore, if $H-v$ is $\ell_{v,c}$-choosable, then we can color $H$ by simply color $v$ with $c$, remove $c$ from the list of its neighbors and color $H-v$. 

\begin{lemma}
Let $v$ be a vertex of $H$ and $f : V(H) \mapsto \mathbb{N}$.
Let $\ell_v : V(H - v) \mapsto \mathbb{N}$ defined by $\ell_v(u) = \ell(u) - \mathbf{1}_{N(v)}$ for every vertex $u \neq v$.
If $H - v$ is $\ell_v$-choosable, then $H$ is $\ell$-choosable.
\end{lemma}

\paragraph{Matrix of inclusion.}
We define a \emph{matrix of inclusion} of $(H, \ell)$ as a square $\{0,1\}$-matrix of size $|V(H)|$ such that, for $u,v\in V(H)$: 
$$M(u,v) = 1 \implies (L(v) \nsubseteq L(u) \implies H \ \text{is} \  \ell \text{-choosable})$$

Note that the identity matrix is a matrix of inclusion since the implication is trivial ($L(u) \nsubseteq L(u)$ being false for every $u$). Moreover, if $\ell(u) < \ell(v)$, then $L(v) \nsubseteq L(u)$. Therefore, we get the following.

\begin{lemma}\label{prop:incl}
Let $M$ be a matrix of inclusion of $(H,\ell)$ and assume that there is a pair $u$ and $v$ such that $M(u,v) = 1$ and $\ell(u) < \ell(v)$. Then $H$ is $\ell$-choosable.
\end{lemma}

\paragraph{The algorithm.}
Let $H$ be a graph and  $\ell : V(H) \mapsto \mathbb{N}$ be a function representing the number of free colors for each vertex. By induction on $H$ and according to the previous paragraphs, if Algorithm~\ref{algo:reduc}, denoted by $\A$, returns \texttt{True}, then $H$ is $\ell$-choosable. 

\begin{algorithm}[!ht]
\KwIn{A graph $H$ and a function $\ell : V(H) \mapsto \mathbb{N}$.}
\KwOut{If it returns \texttt{True}, then $H$ is $\ell$-choosable.}
\BlankLine
\If{$H$ is empty}{\Return{\texttt{True}}}
\If{$H$ has a vertex $v$ such that $\ell(v) \le 0$}{\Return{\texttt{False}} (since $v$ cannot be colored)} 
\If{$H$ has a happy vertex $v$}{\Return{$\A(H - v, \ell)$}}
Compute a matrix of inclusion $M$ of $(H,\ell)$ by induction on $H$ (see corresponding paragraph).

\If{$H$ has a pair of vertices $u$ and $v$ such that $M(u,v) = 1$ and $\ell(u) < \ell(v)$}{\Return{\texttt{True}}}
\Return{\texttt{False}}

\caption{$\mathcal{A}(H, \ell)$}
\label{algo:reduc}
\end{algorithm}

\paragraph{Computing the matrix.} 
The algorithm $\A$ computes a matrix of inclusion $M$ of $(H, \ell)$ by induction on $H$.
As we already noted, we can initialize the matrix $M$ as the identity matrix of size $|V(H)|$.
To update it, the algorithm takes some vertices $u,v$ such that $M(u,v) = 0$ and tests whether it can set $M(u,v) = 1$, \emph{i.e.} $(L(v) \nsubseteq L(u) \implies H \ \text{is} \  \ell \text{-choosable})$. When stabilization is reached, we return the matrix.

We now show how the algorithm proves that $(L(v) \nsubseteq L(u) \implies H \ \text{is} \  \ell \text{-choosable})$ for some pair $u,v$ such that $M(u,v)=0$.

Suppose that $L(v)$ is not included in $L(u)$. We can thus color $v$ with a color $c \notin L(u)$ and then $|L(u)|$ does not decrease. Moreover, for each neighbor $w$ of $v$ such that $M(u,w) = 1$, then we have two cases. Either $c \in L(w)$ so $L(w) \nsubseteq L(u)$ and $H$ is $\ell$-choosable since $M(u,w)=1$. Otherwise, $c \notin L(w)$ and $|L(w)|$ does not decrease either when $v$ gets color $c$. 

Let $\ell' : V(H - v) \mapsto \mathbb{N}$ be the updated lengths of the lists, \emph{i.e.} $\ell'(w)=\ell(w)-1$ when $w$ is a neighbor of $v$ such that $M(u,w)=0$, and $\ell'(w)=\ell(w)$ otherwise. Now, if $H - v$ is $\ell'$-choosable, then $H$ is $\ell$-choosable. In particular, if $\A(H - v, \ell')$ returns \texttt{True}, we can set $M(u,v) = 1$.

\paragraph{Running the program.}
We have applied our heuristic on each of our $40$ configurations. It manages to prove that $28$ of them are reducible. These are depicted in Figure~\ref{fig:theconfigs}. We prove in Section~\ref{sec:remaining} that the $12$ remaining configurations are reducible. Note that even if the algorithm $\A$ fails to reduce some configurations, we still make use of the inclusion matrix it computes to prove that the configurations are indeed reducible.
\input{remaining_configurations/theconfigs}

\subsection{Separating cycles}
\label{sec:sepcyc}
We now tackle auxiliary configurations whose reduction will make the other reductions easier.  Note that each cycle of $G$ separates the plane in three connected components: its \emph{boundary}, its \emph{inside} and its \emph{outside}. We say that $C$ is a \emph{separating cycle} if both its inside and its outside contain at least one vertex of $G$. A vertex of $C$ is \emph{unbalanced} if all its neighbors in $V \setminus C$ are in the same side of $G[V \setminus C]$ and is \emph{balanced} otherwise. A separating cycle is \emph{$k$-balanced} if at most $k$ vertices of $C$ are balanced. 
The goal of this section is to show the following.

\begin{theorem}\label{th:Csep}
The graph $G$ does not contain any $4$-balanced separating cycle of length at most $5$.
\end{theorem}

The rest of this subsection is devoted to prove this statement.
Assume that $G$ contains such a separating cycle $C$. Let us denote by $X= \{ x_1,\ldots,x_\ell \}$ the vertices of $C$ in a cyclic ordering (with $\ell \le 5$) and denote by $Y$ the vertices of $C$ that are balanced (with $|Y| \le 4$ and following again the cyclic ordering of $C$). Note that $|Y| \geq 2$ otherwise $G$ would admit a cut vertex.

Let $G_1$ (resp. $G_2$) be the graph induced by the vertices of $G$ inside (resp. outside) $C$ together with their neighbors in $C$. Note that all the vertices of $G$ lie either in $G_1$ or $G_2$, except for the vertices of $Y$ that appear in both graphs. We claim that the following holds:

\begin{claim}\label{clm:propsep}
It is possible to add edges in $G_1$ and $G_2$ to create two graphs $H_1$ and $H_2$ in such a way that:
\begin{enumerate}
    \item\label{p1} $H_1,H_2\in\mathcal{F}$,
    \item\label{p3} for every pair of vertices $y_i,y_j$ in $Y$, $y_i$ and $y_j$ are adjacent in $H_1$ or $H_2$.
    \item\label{p4} for every pair of vertices $a,b \in C$ that belong to $H_i$ for $i \in \{ 1,2\}$ then $d_{H_i}(a,b) \le d_G(a,b)$.
\end{enumerate}
\end{claim}

\begin{proof}
All along the proof, whenever there exists $3$ consecutive unbalanced vertices $a,b,c$ in $C$ such that $a,c \in G_i$ and $b \notin G_i$
, we add the edge $ac$ in $G_i$ if it does not already exist. We call this operation the saturation operation.

Note that in the graph $G_i$, every vertex $y$ of $Y$ has degree at most $3$. Moreover, its degree is reduced by one more for every additional neighbor of $y$ in $C$ which is not in $G_i$.
We distinguish different cases depending on how many vertices of $Y$ are consecutive in $C$. 
\smallskip

\noindent
\textbf{Case 1.} $C$ contains four consecutive balanced vertices. \\
Without loss of generality, we can assume that $Y=\{ x_1,\ldots,x_4 \}$. By definition, no other vertex of $C$ is in $Y$. 

If $C=Y$ and then, we simply add $x_1x_3$ in one graph and $x_2x_4$ in the second, which satisfies all the points. 
Otherwise, exactly one vertex of $C$ is not in $Y$, w.l.o.g. in $G_1$. Then we simply add $x_1x_3$ in $G_1$ and $x_2x_4, x_1x_4$ in $G_2$. 
\smallskip

\noindent
\textbf{Case 2.} $C$ contains three consecutive balanced vertices. \\
Without loss of generality, we can assume that $Y$ is $x_1,x_2,x_3$ by Case 1.
We can add $x_1x_3$ in both graphs in order to get (\ref{p3}) and (\ref{p4}).
\smallskip

\noindent
\textbf{Case 3.} $C$ contains two consecutive balanced vertices. \\
Without loss of generality, we can assume that $Y$ contains $x_1,x_2$. 

First assume that $Y$ only contains $x_1,x_2$. Then (\ref{p3}) indeed holds. For every $i \le 2$, let $z_i$ (resp. $z_i'$) the smallest (resp. largest) vertex of $C \setminus Y$ that is in $G_i$ (they might not exist and might be the same). For $i \le 2$, we add the edges $z_i'x_1$ and $x_2z_i$ in $H_i$. These edges together with the saturation operation ensures that (\ref{p4}) also holds.

So $C$ is a 5-cycle and the set $Y$ must contain a third vertex which is $x_4$. We add the edge $x_2x_4$ (resp. $x_1x_4$) in the graph $G_i$ which does not contain $x_3$ (resp. $x_5$). So (\ref{p3}) holds. Moreover, if both $x_3,x_5$ belong to the same graph $G_i$, we add $x_1x_4$ to $G_i$ which ensures that (\ref{p4}) also holds. 
\smallskip

\noindent
\textbf{Case 4.} $C$ does not contain consecutive balanced vertices. \\
Without loss of generality, we can assume that $x_1 \in Y$. If it is the only one, the saturation property ensures that the conclusion holds. Otherwise, since $|Y| \ge 2$ (otherwise we would have two consecutive vertices in $Y$), we can assume by symmetry that $Y=\{ x_1,x_3 \}$. 
We now simply add an edge $x_1x_3$ in both graphs to get (\ref{p3}) and (\ref{p4}), which completes the proof. 
\end{proof}

Let us prove that we can conclude using Claim~\ref{clm:propsep}. Let $H_1$ and $H_2$ be the two graphs given by Claim~\ref{clm:propsep}. By minimality, they both admit a distance-$2$ $12$-coloring. We use these colorings and some recoloring steps to construct a distance-2 coloring of $G$ with 12 colors, which ends the proof of Theorem~\ref{th:Csep}. 

Note that in both colorings, the vertices of $Y$ receive distinct colors by (\ref{p4}). So, free to permute colors, we can assume that the vertices of $Y$ are colored alike in both graphs and moreover each vertex $x_i$ is colored $i$. Moreover, we also recolor the vertices of $C\setminus Y$ in such a way that only colors $\{1,\ldots,|C|\}$ appear on $C$.

We first explain that, up to permuting some colors, no vertices of $N(Y)$ colored alike are within distance $2$ in $G$. We will then see how to recolor vertices in $N(C)$ so that we obtain a valid coloring of $G$. Let us denote by $Z$ the vertices $z$ of $N(Y) \setminus C$ with multiplicity (if $z$ is a neighbor of both $y$ and $y'$ we put it twice in $Z$) such that the color of $z$ is not a color of a vertex of $C$. And let $Z_i = Z \cap V(G_i)$ for $i \le 2$. Note that since $|Y| \le 4$, the sizes of both $Z_1$ and $Z_2$ are at most $4$. Also note that only colors $\{1,\ldots,|C|\}$ may appear on $C$, so we can freely use the colors $6,\ldots,12$ (which contains $7$ colors).
\begin{itemize}
    \item If (at least) $2$ vertices of $Z_i$ are colored the same for some $i\le2$, then we recolor them to color $6$. Now we may recolor the at most two remaining vertices of $Z_i$ so that only colors $\{6,7,8\}$ appear on $Z_i$. We may now recolor the vertices of $Z_{3-i}$ with colors from $\{9,10,11,12\}$. 
    \item Otherwise, assume that that $|Z_1|\geqslant|Z_2|$. Then we recolor the vertices of $Z_1$ with pairwise distinct colors from $\{6,\ldots,12\}$ (which is possible since the set has size at most $4$). Now, we can greedily color vertices of $Z_2$ so that each vertex receives a different color from the previously colored vertices in $Z_2$ and from its distance-$2$ neighbors in $Z_1$. Indeed, at each step, at most $4$ colors are forbidden.
\end{itemize}

A \emph{conflict} is a pair of vertices at distance at most $2$ in $G$ that have the same colors. After applying the previous operations, there is no conflict involving vertices in $N(Y)$ anymore. 
There is indeed no conflict between two vertices of $G_1$ (resp. $G_2$). Now consider a pair of vertices in $N(Y)$, one in each $G_i$, and within distance two of each other (so that they share a common neighbor in $Y$). If one of them, called $u$, is colored with the color of some vertex $y$ of $Y$ the other is not since, in at least one of $H_1,H_2$, $u$ is at distance at most $2$ from $y$ by (\ref{p3}).
If the common color is not in $\{ 1,\ldots, 5 \}$, every pair of vertices in $N(Y)$ in $G_1$ and $G_2$ adjacent to the same vertex are colored differently by the above recoloring process and no vertex of $N(Y)$ can be colored with the color of a vertex of $C$ by assumption. So the only remaining conflicts implies some vertex in $N(C)$ and more precisely, for every vertex $x$ of $C \setminus Y$, if $z$ is a neighbor of $x$, then $z$ can only be in conflict with a neighbor of $x$ in $C \setminus Y$ which is in the other graph. In particular, we have $|Y|\leqslant 3$ and $z$ must have a color in $\{1,\ldots,|C|\}$.

Let us now prove that if a vertex $z$ in $N(C)$ has the same color as some $x_i\in Y$ then there is no conflict involving $z$. Indeed, observe that each vertex of $C$ lies within distance $2$ from $x_i$ in either $H_1$ or $H_2$ by (\ref{p3}) and (\ref{p4}), so it cannot be colored with the color of $x_i$, hence of $z$. Moreover, $z$ and $x_i$ cannot be within distance $2$ in $G$, otherwise they would also be within distance $2$ in $H_1$ or $H_2$. Therefore, $z$ may only be in conflict with some vertices in $C\setminus Y$. 

We prevent that by recoloring $N(C\setminus Y)$ with colors larger than $|C|$. Recall that $|Y|\leqslant 3$, hence $|Z_1|\leqslant 3$ and $|Z_2|\leqslant 3$. Therefore, we fixed at most three colors in $H_1$ (resp. $H_2$) corresponding to the vertices in $Z_1$ (resp. $Z_2$). So if there are at most two vertices $x,x'$ of $C \setminus Y$ in $G_1$ (resp. $G_2$), there remain four colors available to recolor their neighbors in $G_1$ (resp. $G_2$) avoiding the colors $\{1,\ldots,|C|\}$ and from $Z_1$ (resp. $Z_2$). 
Finally, there cannot be three vertices in $C \setminus Y$ in $G_1$ and one in $C \setminus Y$ in $G_2$ since $|Y| \ge 2$, which completes the proof of Theorem~\ref{th:Csep}.

\subsection{Remaining configurations}
\label{sec:remaining}
We now reduce the $12$ remaining configurations from $\mathcal{C}$. For brevity, in what follows, when reducing a configuration $C$, we only give the graph $G_C$ (that we denote by $G'$ when $C$ is clear from the context). We will not explicitly check that the adjacencies at distance $2$ are preserved. We only compute the number of forbidden colors for the vertices $v$ of $G$ that are not colored in $G'$.

\begin{lemma}\label{lem:config1}
The graph $G$ does not contain a $3$-vertex incident to a $3$-face.
\end{lemma}
\begin{figure}[hbtp]
        \begin{center}
        \tikzstyle{vertex}=[circle,draw, minimum size=15pt, scale=1, inner sep=1pt]
        \tikzstyle{vertexa}=[circle,draw, minimum size=15pt, scale=1, inner sep=1pt, fill=black!5, very thick]
        \tikzstyle{vertexbc}=[circle,draw, minimum size=15pt, scale=1, inner sep=1pt, fill=black!5]
        \tikzstyle{3vertex}=[draw, minimum size=15pt, scale=1, inner sep=1pt]

        \begin{tikzpicture}[scale=1.2]
            \node (a) at (0.866025404, 0.5) [3vertex] {$a$};
            \node (b) at (0,1) [vertex] {$b$};
            \node (c) at (0,0) [vertex] {$c$};
            \node (d) at (1.866025404,0.5) [vertex] {$d$};
            
            \draw [-, line width=0.03cm] (a) to (b);
            \draw [-, line width=0.03cm] (a) to (c);
            \draw [-, line width=0.03cm] (b) to (c);
            \draw [-, line width=0.03cm] (a) to (d);
            
        \end{tikzpicture}
        \end{center}
        \caption{Forbidden configuration of Lemma~\ref{lem:config1}}
        \label{fig:sconf1}
    \end{figure}
\begin{proof}
Assume that $G$ contains a $3$-vertex $a$ incident to a $3$-face $[a, b, c]$ and a neighbor $d$ (see Figure~\ref{fig:sconf1} for an illustration). Define $G' = G - a + bd\in\mathcal{F}$. By minimality, $G'$ admits a distance-$2$ $12$-coloring. This coloring is a partial coloring of $G$ and forbids at most $4 + 2\cdot3 = 10 < 12$ colors for $a$. So it can be extended into a distance-$2$ $12$-coloring of $G$, which completes the proof.
\end{proof}

\begin{lemma}\label{lem:config2}
The graph $G$ does not contain a $4$-vertex incident with two $3$-faces and a $4$-face, nor to a $3$-face and three $4$-faces.
\end{lemma}
\begin{figure}[hbtp]
        \begin{center}
        \tikzstyle{vertex}=[circle,draw, minimum size=15pt, scale=1, inner sep=1pt]
        \tikzstyle{vertexa}=[circle,draw, minimum size=15pt, scale=1, inner sep=1pt, fill=black!5, very thick]
        \tikzstyle{vertexbc}=[circle,draw, minimum size=15pt, scale=1, inner sep=1pt, fill=black!5]
        \tikzstyle{3vertex}=[draw, minimum size=15pt, scale=1, inner sep=1pt]

        \begin{tikzpicture}[scale=1.2]
            \node (a) at (0, 0) [vertex] {$a$};
            \node (b) at (150:0.707) [vertex] {$b$};
            \node (c) at (210:0.707) [vertex] {$c$};
            \node (d) at (30:0.707) [vertex] {$d$};
            \node (e) at (-30:0.707) [vertex] {$e$};
            \node (f) at ($(c)+(-30:0.707)$) [vertex] {$f$};

            \draw [-, line width=0.03cm] (a) to (b);
            \draw [-, line width=0.03cm] (a) to (c);
            \draw [-, line width=0.03cm] (b) to (c);
            \draw [-, line width=0.03cm] (a) to (d);
            \draw [-, line width=0.03cm] (a) to (e);
            \draw [-, line width=0.03cm] (e) to (d);
            \draw [-, line width=0.03cm] (e) to (f);
            \draw [-, line width=0.03cm] (c) to (f);
            
        \end{tikzpicture}
        \begin{tikzpicture}[scale=1.2]
            \node (a) at (0, 0) [vertex] {$a$};
            \node (b) at (135:0.707) [vertex] {$b$};
            \node (c) at (225:0.707) [vertex] {$d$};
            \node (d) at (45:0.707) [vertex] {$c$};
            \node (e) at (-45:0.707) [vertex] {$e$};
            \node (f) at (0,-1) [vertex] {$f$};
            \node (g) at (-1,0) [vertex] {$g$};
            \node (h) at (1,0) [vertex] {$h$};

            \draw [-, line width=0.03cm] (a) to (b);
            \draw [-, line width=0.03cm] (a) to (c);
            \draw [-, line width=0.03cm] (g) to (c);
            \draw [-, line width=0.03cm] (g) to (b);
            \draw [-, line width=0.03cm] (h) to (e);
            \draw [-, line width=0.03cm] (h) to (d);
            \draw [-, line width=0.03cm] (b) to (d);
            \draw [-, line width=0.03cm] (a) to (d);
            \draw [-, line width=0.03cm] (a) to (e);
            \draw [-, line width=0.03cm] (e) to (f);
            \draw [-, line width=0.03cm] (c) to (f);
            
        \end{tikzpicture}
        \end{center}
        \caption{Forbidden configurations  of Lemma~\ref{lem:config2}}
        \label{fig:sconf2}
    \end{figure}
\begin{proof}
Assume that $G$ contains such a vertex. Using that $G$ does not contain any $3$-face edge-incident to a $3$-face and a $4$-face (this is reduced by the heuristic), we are in one of the situations depicted in Figure~\ref{fig:sconf2}. In both cases, define $G' = G - a + bd + ce\in\mathcal{F}$. By minimality, it admits a distance-$2$ $12$-coloring.
This coloring is a partial coloring of $G$ and forbids at most $11$ colors for $a$, so it can be extended into a distance-$2$ $12$-coloring of $G$, which completes the proof.
\end{proof}

\begin{lemma}\label{lem:config3}
The graph $G$ does not contain a $3$-vertex incident to two $4$-faces.
\end{lemma}
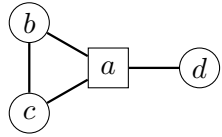
\begin{figure}[hbtp]
        \begin{center}
        \tikzstyle{vertex}=[circle,draw, minimum size=15pt, scale=1, inner sep=1pt]
        \tikzstyle{vertexa}=[circle,draw, minimum size=15pt, scale=1, inner sep=1pt, fill=black!5, very thick]
        \tikzstyle{vertexbc}=[circle,draw, minimum size=15pt, scale=1, inner sep=1pt, fill=black!5]
        \tikzstyle{3vertex}=[draw, minimum size=15pt, scale=1, inner sep=1pt]
        \tikzstyle{vertexcolor}=[circle,draw, minimum size=7pt, scale=1, inner sep=1pt, fill=black]

        \begin{tikzpicture}[scale=1.2]
            \node (a) at (0,0) [3vertex] {$a$};
            \node (b) at (0, 1) [vertex] {$d$};
            \node (c) at (-1, 0) [vertex] {$b$};
            \node (d) at (-1,1) [vertex] {$c$};
            \node (e) at (1, 0) [vertex] {$f$};
            \node (f) at (1, 1) [vertex] {$e$};
            
            \draw [-, line width=0.03cm] (a) to (b);
            \draw [-, line width=0.03cm] (b) to (d);
            \draw [-, line width=0.03cm] (a) to (c);
            \draw [-, line width=0.03cm] (c) to (d);
            \draw [-, line width=0.03cm] (a) to (e);
            \draw [-, line width=0.03cm] (e) to (f);
            \draw [-, line width=0.03cm] (f) to (b);

        \end{tikzpicture}
        \end{center}
        \caption{Forbidden configuration  of Lemma~\ref{lem:config3}}
        \label{fig:sconf3}
    \end{figure}
\begin{proof}
Assume that $G$ contains a $3$-vertex $a$ incident to two $4$-faces $[a,b,c,d]$ and $[a,d,e,f]$ (see Figure~\ref{fig:sconf3} for an illustration). Define $G' = G - a + bf\in\mathcal{F}$. By minimality, it admits a distance-$2$ $12$-coloring. 

This coloring is a partial coloring of $G$ and forbids at most $2\cdot3 + 2 + 1 + 1 = 10 < 12$ colors for $a$, so it can be extended into a distance-$2$ $12$-coloring, which completes the proof.
\end{proof}

The eight remaining configurations $C_0$ to $C_7$ are respectively illustrated by Figures~\ref{fig:sconf4} to~\ref{fig:conf4}. In each drawing, $3$-vertices are denoted by squares and $4^-$-vertices by circles (i.e. we do not know if they have degree $3$ or $4$). On the right, each vertex is either black, meaning we do not have to color it (these vertices will be in $G'$), or it is labeled with the number of available colors it has.

\begin{lemma}\label{lem:configR0}
Configuration $C_0$ is reducible.
\end{lemma}
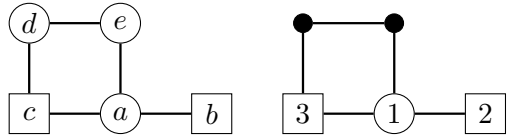
\begin{figure}[hbtp]
        \begin{center}
        \tikzstyle{vertex}=[circle,draw, minimum size=15pt, scale=1, inner sep=1pt]
        \tikzstyle{vertexa}=[circle,draw, minimum size=15pt, scale=1, inner sep=1pt, fill=black!5, very thick]
        \tikzstyle{vertexbc}=[circle,draw, minimum size=15pt, scale=1, inner sep=1pt, fill=black!5]
        \tikzstyle{3vertex}=[draw, minimum size=15pt, scale=1, inner sep=1pt, ]
        \tikzstyle{vertexcolor}=[circle,draw, minimum size=7pt, scale=1, inner sep=1pt, fill=black]

        \begin{tikzpicture}[scale=1.2]
            \node (a) at (0,0) [vertex] {$a$};
            \node (b) at (0, 1) [vertex] {$e$};
            \node (c) at (-1, 0) [3vertex] {$c$};
            \node (d) at (-1,1) [vertex] {$d$};
            \node (e) at (1, 0) [3vertex] {$b$};
            
            \draw [-, line width=0.03cm] (a) to (b);
            \draw [-, line width=0.03cm] (b) to (d);
            \draw [-, line width=0.03cm] (a) to (c);
            \draw [-, line width=0.03cm] (c) to (d);
            \draw [-, line width=0.03cm] (a) to (e);
            
            \node (a1) at (3,0) [vertex] {$1$};
            \node (b1) at (3, 1) [vertexcolor] {};
            \node (c1) at (2, 0) [3vertex] {$3$};
            \node (d1) at (2,1) [vertexcolor] {};
            \node (e1) at (4, 0) [3vertex] {$2$};
            
            \draw [-, line width=0.03cm] (a1) to (b1);
            \draw [-, line width=0.03cm] (b1) to (d1);
            \draw [-, line width=0.03cm] (a1) to (c1);
            \draw [-, line width=0.03cm] (c1) to (d1);
            \draw [-, line width=0.03cm] (a1) to (e1);
            
        \end{tikzpicture}
        \end{center}
        \caption{Configuration $C_0$}
        \label{fig:sconf4}
    \end{figure}
\begin{proof}
Assume that $G$ contains $C_0$. We use the notation of Figure~\ref{fig:sconf4}. Let $G'=G-ac\in\mathcal{F}$. By minimality, it admits a distance-$2$ $12$-coloring $\alpha$. Uncolor $a,b,c$. Now observe that no two colored vertices of $G$ within distance $2$ have the same color under $\alpha$.

For each vertex $v\in\{a,b,c\}$, let $L(v)$ be the set of free colors of $v$ and let $\ell(v)$ be the minimum possible size of $L(v)$ (see Figure~\ref{fig:sconf4} for an illustration). We have to show that the graph induced by $a,b,c$ is $\ell$-choosable. To this end, observe that we may color $a,b,c$ in that order. This yields a distance-$2$ $12$-coloring of $G$, which completes the proof.
\end{proof}

To reduce configurations $C_1$ to $C_6$ depicted on Figures~\ref{fig:conf1} to~\ref{fig:conf5}, we first apply Theorem~\ref{th:Csep} to prove that some vertices cannot be within distance $2$ in a graph containing them. This will then color them with the same color if they have a common available color, so that their common neighbors lose only one color in their lists. The following claim can be easily checked on each of the configurations: 
    \begin{figure}[hbtp]
        \begin{center}
        \tikzstyle{vertex}=[circle,draw, minimum size=15pt, scale=1, inner sep=1pt]
        \tikzstyle{vertexa}=[circle,draw, minimum size=15pt, scale=1, inner sep=1pt, fill=black!5, very thick]
        \tikzstyle{vertexbc}=[circle,draw, minimum size=15pt, scale=1, inner sep=1pt, fill=black!5]
        
        \begin{tikzpicture}[scale=0.8]
            \node (a) at (-0.5, 0.866025404) [vertex] {$g$};
            \node (b) at (0,0) [vertexa] {$a$};
            \node (c) at (1,0) [vertex] {$d$};
            \node (d) at (-1,0) [vertexbc] {$b$};
            \node (e) at (-1,-1) [vertex] {$e$};
            \node (f) at (0,-1) [vertex] {$f$};
            \node (g) at (1,-1) [vertexbc] {$c$};
            \node (h) at (1.866025404,-0.5) [vertex] {$h$};
         
            \draw [-, line width=0.03cm] (a) to (b);
            \draw [-, line width=0.03cm] (a) to (d);
            \draw [-, line width=0.03cm] (b) to (c);
            \draw [-, line width=0.03cm] (b) to (d);
            \draw [-, line width=0.03cm] (d) to (e);
            \draw [-, line width=0.03cm] (e) to (f);
            \draw [-, line width=0.03cm] (f) to (b);
            \draw [-, line width=0.03cm] (f) to (g);
            \draw [-, line width=0.03cm] (g) to (c);
            \draw [-, line width=0.03cm] (g) to (h);
            \draw [-, line width=0.03cm] (h) to (c);
            
            \node (a1) at (5.5, 0.866025404) [vertex] {$3$};
            \node (b1) at (6,0) [vertexa] {$7$};
            \node (c1) at (7,0) [vertex] {$5$};
            \node (d1) at (5,0) [vertexbc] {$4$};
            \node (e1) at (5,-1) [vertex] {$2$};
            \node (f1) at (6,-1) [vertex] {$5$};
            \node (g1) at (7,-1) [vertexbc] {$4$};
            \node (h1) at (7.866025404,-0.5) [vertex] {$2$};
         
            \draw [-, line width=0.03cm] (a1) to (b1);
            \draw [-, line width=0.03cm] (a1) to (d1);
            \draw [-, line width=0.03cm] (b1) to (c1);
            \draw [-, line width=0.03cm] (b1) to (d1);
            \draw [-, line width=0.03cm] (d1) to (e1);
            \draw [-, line width=0.03cm] (e1) to (f1);
            \draw [-, line width=0.03cm] (f1) to (b1);
            \draw [-, line width=0.03cm] (f1) to (g1);
            \draw [-, line width=0.03cm] (g1) to (c1);
            \draw [-, line width=0.03cm] (g1) to (h1);
            \draw [-, line width=0.03cm] (h1) to (c1);
        
        \end{tikzpicture}
        \end{center}
        \caption{Configuration $C_1$}
        \label{fig:conf1}
    \end{figure}
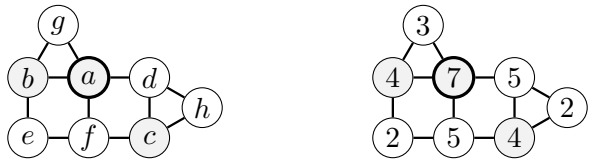
    
    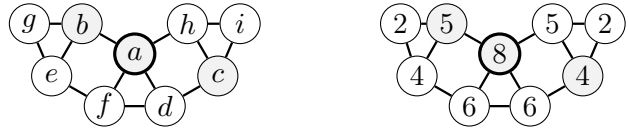
\begin{figure}[hbtp]
        \begin{center}
        \tikzstyle{vertex}=[circle,draw, minimum size=15pt, scale=1, inner sep=1pt]
        \tikzstyle{vertexa}=[circle,draw, minimum size=15pt, scale=1, inner sep=1pt, fill=black!5, very thick]
        \tikzstyle{vertexbc}=[circle,draw, minimum size=15pt, scale=1, inner sep=1pt, fill=black!5]
        
        \begin{tikzpicture}[scale=0.8]
            \node (a) at (-0.866025404, 0.5) [vertexbc] {$b$};
            \node (b) at (0,0) [vertexa] {$a$};
            \node (c) at (0.866025404,0.5) [vertex] {$h$};
            \node (d) at (-1.732050808,0.5) [vertex] {$g$};
            \node (e) at (-1.366025404, -0.366025404) [vertex] {$e$};
            \node (f) at (-0.5,- 0.866025404) [vertex] {$f$};
            \node (g) at (0.5,- 0.866025404) [vertex] {$d$};
            \node (h) at (1.366025404, -0.366025404) [vertexbc] {$c$};
            \node (i) at (1.732050808,0.5) [vertex] {$i$};
         
            \draw [-, line width=0.03cm] (a) to (b);
            \draw [-, line width=0.03cm] (a) to (d);
            \draw [-, line width=0.03cm] (b) to (c);
            \draw [-, line width=0.03cm] (d) to (e);
            \draw [-, line width=0.03cm] (e) to (f);
            \draw [-, line width=0.03cm] (f) to (b);
            \draw [-, line width=0.03cm] (f) to (g);
            \draw [-, line width=0.03cm] (c) to (i);
            \draw [-, line width=0.03cm] (g) to (h);
            \draw [-, line width=0.03cm] (h) to (c);
            \draw [-, line width=0.03cm] (h) to (i);
            \draw [-, line width=0.03cm] (b) to (g);
            \draw [-, line width=0.03cm] (a) to (e);
            
            \node (a1) at (5.13, 0.5) [vertexbc] {5};
            \node (b1) at (6,0) [vertexa] {8};
            \node (c1) at (6.866025404,0.5) [vertex] {5};
            \node (d1) at (4.37,0.5) [vertex] {2};
            \node (e1) at (4.64, -0.366025404) [vertex] {4};
            \node (f1) at (5.5,- 0.866025404) [vertex] {6};
            \node (g1) at (6.5,- 0.866025404) [vertex] {6};
            \node (h1) at (7.366025404, -0.366025404) [vertexbc] {4};
            \node (i1) at (7.732050808,0.5) [vertex] {2};
         
            \draw [-, line width=0.03cm] (a1) to (b1);
            \draw [-, line width=0.03cm] (a1) to (d1);
            \draw [-, line width=0.03cm] (b1) to (c1);
            \draw [-, line width=0.03cm] (d1) to (e1);
            \draw [-, line width=0.03cm] (e1) to (f1);
            \draw [-, line width=0.03cm] (f1) to (b1);
            \draw [-, line width=0.03cm] (f1) to (g1);
            \draw [-, line width=0.03cm] (c1) to (i1);
            \draw [-, line width=0.03cm] (g1) to (h1);
            \draw [-, line width=0.03cm] (h1) to (c1);
            \draw [-, line width=0.03cm] (h1) to (i1);
            \draw [-, line width=0.03cm] (b1) to (g1);
            \draw [-, line width=0.03cm] (a1) to (e1);

        \end{tikzpicture}
        \end{center}
        \caption{Configuration $C_2$}
        \label{fig:conf2}
    \end{figure}
    
    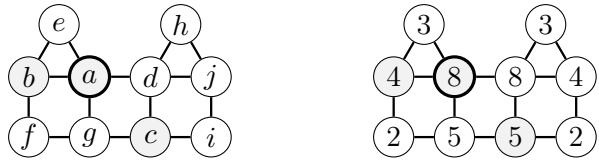
\begin{figure}[hbtp]
        \begin{center}
        \tikzstyle{vertex}=[circle,draw, minimum size=15pt, scale=1, inner sep=1pt]
        \tikzstyle{vertexa}=[circle,draw, minimum size=15pt, scale=1, inner sep=1pt, fill=black!5, very thick]
        \tikzstyle{vertexbc}=[circle,draw, minimum size=15pt, scale=1, inner sep=1pt, fill=black!5]
        
       \begin{tikzpicture}[scale=0.8]
            \node (a) at (-0.5, 0.866025404) [vertex] {$e$};
            \node (b) at (0,0) [vertexa] {$a$};
            \node (c) at (1,0) [vertex] {$d$};
            \node (d) at (1.5, 0.866025404) [vertex] {$h$};
            \node (e) at (-1,0) [vertexbc] {$b$};
            \node (f) at (-1,-1) [vertex] {$f$};
            \node (g) at (0,-1) [vertex] {$g$};
            \node (h) at (1,-1) [vertexbc] {$c$};
            \node (i) at (2, -1) [vertex] {$i$};
            \node (j) at (2, 0) [vertex] {$j$};
           
            \draw [-, line width=0.03cm] (a) to (b);
            \draw [-, line width=0.03cm] (a) to (e);
            \draw [-, line width=0.03cm] (b) to (c);
            \draw [-, line width=0.03cm] (c) to (d);
            \draw [-, line width=0.03cm] (d) to (j);
            \draw [-, line width=0.03cm] (e) to (f);
            \draw [-, line width=0.03cm] (g) to (b);
            \draw [-, line width=0.03cm] (f) to (g);
            \draw [-, line width=0.03cm] (b) to (e);
            \draw [-, line width=0.03cm] (g) to (h);
            \draw [-, line width=0.03cm] (h) to (c);
            \draw [-, line width=0.03cm] (h) to (i);
            \draw [-, line width=0.03cm] (i) to (j);
            \draw [-, line width=0.03cm] (j) to (c);
            
            \node (a1) at (5.5, 0.866025404) [vertex] {3};
            \node (b1) at (6,0) [vertexa] {8};
            \node (c1) at (7,0) [vertex] {8};
            \node (d1) at (7.5, 0.866025404) [vertex] {3};
            \node (e1) at (5,0) [vertexbc] {4};
            \node (f1) at (5,-1) [vertex] {2};
            \node (g1) at (6,-1) [vertex] {5};
            \node (h1) at (7,-1) [vertexbc] {5};
            \node (i1) at (8, -1) [vertex] {2};
            \node (j1) at (8, 0) [vertex] {4};
           
            \draw [-, line width=0.03cm] (a1) to (b1);
            \draw [-, line width=0.03cm] (a1) to (e1);
            \draw [-, line width=0.03cm] (b1) to (c1);
            \draw [-, line width=0.03cm] (c1) to (d1);
            \draw [-, line width=0.03cm] (d1) to (j1);
            \draw [-, line width=0.03cm] (e1) to (f1);
            \draw [-, line width=0.03cm] (g1) to (b1);
            \draw [-, line width=0.03cm] (f1) to (g1);
            \draw [-, line width=0.03cm] (b1) to (e1);
            \draw [-, line width=0.03cm] (g1) to (h1);
            \draw [-, line width=0.03cm] (h1) to (c1);
            \draw [-, line width=0.03cm] (h1) to (i1);
            \draw [-, line width=0.03cm] (i1) to (j1);
            \draw [-, line width=0.03cm] (j1) to (c1);

        \end{tikzpicture}
        \end{center}
        \caption{Configuration $C_3$}
        \label{fig:conf3}
    \end{figure}
    
    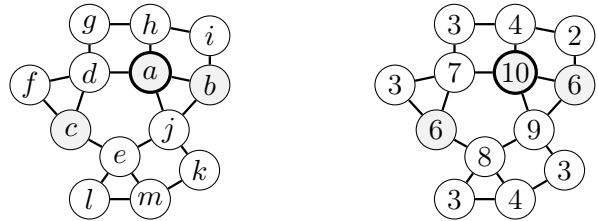
\begin{figure}[hbtp]
        \begin{center}
        \tikzstyle{vertex}=[circle, draw, minimum size=15pt, scale=1, inner sep=1pt]
        \tikzstyle{vertexa}=[circle,draw, minimum size=15pt, scale=1, inner sep=1pt, fill=black!5, very thick]
        \tikzstyle{vertexbc}=[circle,draw, minimum size=15pt, scale=1, inner sep=1pt, fill=black!5]
        
           \begin{tikzpicture}[scale=0.8]
            \node (a) at (-0.30901699437, -0.95105651629) [vertexbc] {$c$};
            \node (b) at (0,0) [vertex] {$d$};
            \node (c) at (1,0) [vertexa] {$a$};
            \node (d) at (1.30901699437, -0.95105651629) [vertex] {$j$};
            \node (e) at (0.5, -1.39680224667) [vertex] {$e$};
            \node (f) at (-168:1) [vertex] {$f$};
            \node (g) at (0, 0.8) [vertex] {$g$};
            \node (h) at (1, 0.8) [vertex]  {$h$};
            \node (i) at (1.97814760073, 0.60791169081) [vertex] {$i$};
            \node (j) at (1.97814760073, -0.20791169081) [vertexbc] {$b$};
            \node (k) at (1.80901699437, -1.62) [vertex]  {$k$};
            \node (l) at (0, -2.1) [vertex] {$l$};
            \node (m) at (1, -2.1) [vertex] {$m$};

            \draw [-, line width=0.03cm] (a) to (b);
            \draw [-, line width=0.03cm] (a) to (e);
            \draw [-, line width=0.03cm] (b) to (c);
            \draw [-, line width=0.03cm] (c) to (d);
            \draw [-, line width=0.03cm] (d) to (e);
            \draw [-, line width=0.03cm] (c) to (j);
            \draw [-, line width=0.03cm] (h) to (i);
            \draw [-, line width=0.03cm] (f) to (b);
            \draw [-, line width=0.03cm] (g) to (h);
            \draw [-, line width=0.03cm] (g) to (b);
            \draw [-, line width=0.03cm] (h) to (c);
            \draw [-, line width=0.03cm] (j) to (d);
            \draw [-, line width=0.03cm] (f) to (a);
            \draw [-, line width=0.03cm] (i) to (j);
            \draw [-, line width=0.03cm] (e) to (m);
            \draw [-, line width=0.03cm] (e) to (l);
            \draw [-, line width=0.03cm] (m) to (l);
            \draw [-, line width=0.03cm] (k) to (d);
            \draw [-, line width=0.03cm] (k) to (m);
            
             \node (a1) at (5.7, -0.95105651629) [vertexbc] {6};
            \node (b1) at (6,0) [vertex] {7};
            \node (c1) at (7,0) [vertexa] {10};
            \node (d1) at (7.30901699437, -0.95105651629) [vertex] {9};
            \node (e1) at (6.5, -1.39680224667) [vertex] {8};
            \node (f1) at (5.02185239927, -0.20791169081) [vertex] {3};
            \node (g1) at (6, 0.8) [vertex] {3};
            \node (h1) at (7, 0.8) [vertex]  {4};
            \node (i1) at (7.97814760073, 0.60791169081) [vertex] {2};
            \node (j1) at (7.97814760073, -0.20791169081) [vertexbc] {6};
            \node (k1) at (7.80901699437, -1.62) [vertex]  {3};
            \node (l1) at (6, -2.1) [vertex] {3};
            \node (m1) at (7, -2.1) [vertex] {4};

            \draw [-, line width=0.03cm] (a1) to (b1);
            \draw [-, line width=0.03cm] (a1) to (e1);
            \draw [-, line width=0.03cm] (b1) to (c1);
            \draw [-, line width=0.03cm] (c1) to (d1);
            \draw [-, line width=0.03cm] (d1) to (e1);
            \draw [-, line width=0.03cm] (c1) to (j1);
            \draw [-, line width=0.03cm] (h1) to (i1);
            \draw [-, line width=0.03cm] (f1) to (b1);
            \draw [-, line width=0.03cm] (g1) to (h1);
            \draw [-, line width=0.03cm] (g1) to (b1);
            \draw [-, line width=0.03cm] (h1) to (c1);
            \draw [-, line width=0.03cm] (j1) to (d1);
            \draw [-, line width=0.03cm] (f1) to (a1);
            \draw [-, line width=0.03cm] (i1) to (j1);
            \draw [-, line width=0.03cm] (e1) to (m1);
            \draw [-, line width=0.03cm] (e1) to (l1);
            \draw [-, line width=0.03cm] (m1) to (l1);
            \draw [-, line width=0.03cm] (k1) to (d1);
            \draw [-, line width=0.03cm] (k1) to (m1);

        \end{tikzpicture}
        \end{center}
        \caption{Configuration $C_4$}
        \label{fig:conf9}
    \end{figure}

        
    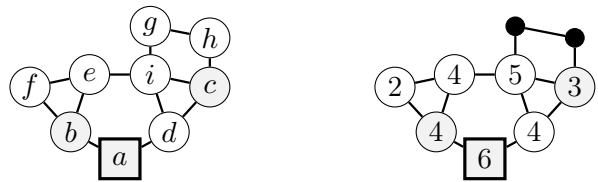
\begin{figure}[hbtp]
        \begin{center}
        \tikzstyle{vertex}=[circle, draw, minimum size=15pt, scale=1, inner sep=1pt]
        \tikzstyle{3vertex}=[draw, minimum size=15pt, scale=1, inner sep=1pt,  fill=black!5, very thick]
        \tikzstyle{vertexa}=[circle,draw, minimum size=15pt, scale=1, inner sep=1pt, fill=black!5, very thick]
        \tikzstyle{vertexbc}=[circle,draw, minimum size=15pt, scale=1, inner sep=1pt, fill=black!5]
        \tikzstyle{vertexcolor}=[circle,draw, minimum size=7pt, scale=1, inner sep=1pt, fill=black]
        
       \begin{tikzpicture}[scale=0.8]
            \node (a) at (-0.30901699437, -0.95105651629) [vertexbc] {$b$};
            \node (b) at (0,0) [vertex] {$e$};
            \node (c) at (1,0) [vertex] {$i$};
            \node (d) at (1.30901699437, -0.95105651629) [vertex] {$d$};
            \node (e) at (0.5, -1.39680224667) [3vertex] {$a$};
            \node (f) at (-168:1) [vertex] {$f$};
            \node (g) at (1, 0.8) [vertex]  {$g$};
            \node (h) at (1.97814760073, 0.60791169081) [vertex] {$h$};
            \node (i) at (1.97814760073, -0.20791169081) [vertexbc] {$c$};
            
            \draw [-, line width=0.03cm] (a) to (b);
            \draw [-, line width=0.03cm] (a) to (e);
            \draw [-, line width=0.03cm] (b) to (c);
            \draw [-, line width=0.03cm] (c) to (d);
            \draw [-, line width=0.03cm] (d) to (e);
            \draw [-, line width=0.03cm] (h) to (i);
            \draw [-, line width=0.03cm] (f) to (b);
            \draw [-, line width=0.03cm] (g) to (h);
            \draw [-, line width=0.03cm] (g) to (c);
            \draw [-, line width=0.03cm] (i) to (c);
            \draw [-, line width=0.03cm] (i) to (d);
            \draw [-, line width=0.03cm] (f) to (a);
            
            \node (a1) at (5.7, -0.95105651629) [vertexbc] {4};
            \node (b1) at (6,0) [vertex] {4};
            \node (c1) at (7,0) [vertex] {5};
            \node (d1) at (7.30901699437, -0.95105651629) [vertex] {4};
            \node (e1) at (6.5, -1.39680224667) [3vertex] {6};
            \node (f1) at (5.02185239927, -0.20791169081) [vertex] {2};
            \node (g1) at (7, 0.8) [vertexcolor]  {};
            \node (h1) at (7.97814760073, 0.60791169081) [vertexcolor] {};
            \node (i1) at (7.97814760073, -0.20791169081) [vertexbc] {3};
            
            \draw [-, line width=0.03cm] (a1) to (b1);
            \draw [-, line width=0.03cm] (a1) to (e1);
            \draw [-, line width=0.03cm] (b1) to (c1);
            \draw [-, line width=0.03cm] (c1) to (d1);
            \draw [-, line width=0.03cm] (d1) to (e1);
            \draw [-, line width=0.03cm] (h1) to (i1);
            \draw [-, line width=0.03cm] (f1) to (b1);
            \draw [-, line width=0.03cm] (g1) to (h1);
            \draw [-, line width=0.03cm] (g1) to (c1);
            \draw [-, line width=0.03cm] (i1) to (c1);
            \draw [-, line width=0.03cm] (i1) to (d1);
            \draw [-, line width=0.03cm] (f1) to (a1);

        \end{tikzpicture}
        \end{center}
        \caption{Configuration $C_5$}
        \label{fig:conf8}
    \end{figure}
    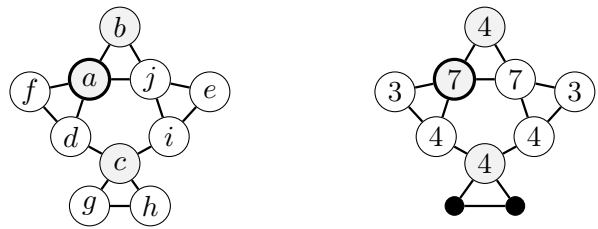
\begin{figure}[hbtp]
        \begin{center}
        \tikzstyle{vertex}=[circle, draw, minimum size=15pt, scale=1, inner sep=1pt]
         \tikzstyle{vertexa}=[circle,draw, minimum size=15pt, scale=1, inner sep=1pt, fill=black!5, very thick]
        \tikzstyle{vertexbc}=[circle,draw, minimum size=15pt, scale=1, inner sep=1pt, fill=black!5]
        \tikzstyle{vertexcolor}=[circle,draw, minimum size=7pt, scale=1, inner sep=1pt, fill=black]
        
       \begin{tikzpicture}[scale=0.8]
            \node (a) at (-0.30901699437, -0.95105651629) [vertex] {$d$};
            \node (b) at (0,0) [vertexa] {$a$};
            \node (c) at (1,0) [vertex] {$j$};
            \node (d) at (1.30901699437, -.95105651629) [vertex] {$i$};
            \node (e) at (0.5, -1.39680224667) [vertexbc] {$c$};
            \node (f) at (-168:1) [vertex] {$f$};
            \node (g) at (0.5, 0.866025404) [vertexbc] {$b$};
            \node (h) at (1.97814760073, -0.20791169081) [vertex]  {$e$};
            \node (i) at (0, -2.1) [vertex] {$g$};
            \node (j) at (1, -2.1) [vertex] {$h$};

            \draw [-, line width=0.03cm] (a) to (b);
            \draw [-, line width=0.03cm] (a) to (e);
            \draw [-, line width=0.03cm] (b) to (c);
            \draw [-, line width=0.03cm] (c) to (d);
            \draw [-, line width=0.03cm] (d) to (e);
            \draw [-, line width=0.03cm] (e) to (j);
            \draw [-, line width=0.03cm] (e) to (i);
            \draw [-, line width=0.03cm] (f) to (b);
            \draw [-, line width=0.03cm] (g) to (c);
            \draw [-, line width=0.03cm] (g) to (b);
            \draw [-, line width=0.03cm] (h) to (c);
            \draw [-, line width=0.03cm] (h) to (d);
            \draw [-, line width=0.03cm] (f) to (a);
            \draw [-, line width=0.03cm] (i) to (j);
            
            \node (a1) at (5.7, -0.95105651629) [vertex] {4};
            \node (b1) at (6,0) [vertexa] {7};
            \node (c1) at (7,0) [vertex] {7};
            \node (d1) at (7.30901699437, -.95105651629) [vertex] {4};
            \node (e1) at (6.5, -1.39680224667) [vertexbc] {4};
            \node (f1) at (5.02185239927, -0.20791169081) [vertex] {3};
            \node (g1) at (6.5, 0.866025404) [vertexbc] {4};
            \node (h1) at (7.97814760073, -0.20791169081) [vertex]  {3};
            \node (i1) at (6, -2.1) [vertexcolor] {};
            \node (j1) at (7, -2.1) [vertexcolor] {};

            \draw [-, line width=0.03cm] (a1) to (b1);
            \draw [-, line width=0.03cm] (a1) to (e1);
            \draw [-, line width=0.03cm] (b1) to (c1);
            \draw [-, line width=0.03cm] (c1) to (d1);
            \draw [-, line width=0.03cm] (d1) to (e1);
            \draw [-, line width=0.03cm] (e1) to (j1);
            \draw [-, line width=0.03cm] (e1) to (i1);
            \draw [-, line width=0.03cm] (f1) to (b1);
            \draw [-, line width=0.03cm] (g1) to (c1);
            \draw [-, line width=0.03cm] (g1) to (b1);
            \draw [-, line width=0.03cm] (h1) to (c1);
            \draw [-, line width=0.03cm] (h1) to (d1);
            \draw [-, line width=0.03cm] (f1) to (a1);
            \draw [-, line width=0.03cm] (i1) to (j1);

        \end{tikzpicture}
        \end{center}
        \caption{Configuration $C_6$}
        \label{fig:conf5}
    \end{figure}

        

\begin{claim}
\label{claim:bc}
If $G$ contains a configuration $C \in \{ C_1, \ldots,C_6\}$, then $b$ and $c$ are at distance at least $3$ in $G$. 
\end{claim}

We will need the following definition to conclude the reductions.  Given a graph $H$ and a function $\ell : V(H) \mapsto \mathbb{N}$, we say that $H$ is \emph{happy} if $H$ is empty or $H$ contains a happy vertex $v$ such that $H\setminus v$ is happy. Observe that if $H$ is happy then $H$ is $\ell$-choosable since we can color greedily its vertices from the last happy vertex to the first. Note that even if the algorithm $\A$ did not conclude to the reducibility, we make use the matrix of inclusions it computed.

\begin{lemma}\label{lem:configR1to4} 
Configurations $C_1$ to $C_6$ are reducible.
\end{lemma}

\begin{proof}
Assume that $G$ contains such a configuration $C$. Define $G_C = G - a$, which admits a distance-$2$ $12$-coloring by minimality. Uncolor all vertices $v$ of $C$, except in $C\in\{C_5,C_6\}$ where we keep $g,h$ colored. Let $H$ be the subgraph of $G^2$ induced by the uncolored vertex of $C$. For each vertex of $H$, denote by $L(v)$ be the set of free colors for $v$ and let $\ell(v)$ be the minimum possible size of $L(v)$ (see Figures~\ref{fig:conf1} to~\ref{fig:conf5}). 

Applying the algorithm $\A$ on $(H,\ell)$, we obtain that $H$ is $\ell$-choosable unless every vertex $v$ within distance $2$ from $a$ satisfies $L(v) \subseteq L(a)$. 

In particular, we have $L(b) \cup L(c) \subseteq L(a)$ and $|L(b) \cup L(c)| > |L(a)|$, hence $b$ and $c$ have a common free color. By Claim~\ref{claim:bc}, $b$ and $c$ are not neighbors in $G^2$, so we can color them both with that common free color. Afterwards, one can easily check that the graph $H - b - c$ is happy, which provides a distance-$2$ $12$-coloring of $G$ and concludes the proof.
\end{proof}

We end this section with the reduction of the last configuration $C_7$. 

\begin{lemma}
Configuration $C_7$ is reducible.
\end{lemma}
    \begin{figure}[hbtp]
        \begin{center}
        \tikzstyle{vertex}=[circle, draw, minimum size=15pt, scale=1, inner sep=1pt]
        \tikzstyle{vertexa}=[circle,draw, minimum size=15pt, scale=1, inner sep=1pt, fill=black!5, very thick]
        \tikzstyle{vertexbc}=[circle,draw, minimum size=15pt, scale=1, inner sep=1pt, fill=black!5]
        
       \begin{tikzpicture}[scale=0.8]
            \node (a) at (-0.30901699437, -0.95105651629) [vertexbc] {$a$};
            \node (b) at (0,0) [vertexa] {$b$};
            \node (c) at (1,0) [vertex] {$c$};
            \node (d) at (1.30901699437, -.95105651629) [vertex] {$d$};
            \node (e) at (0.5, -1.39680224667) [vertex] {$e$};
            \node (f) at (-1.17504239816,  -0.45105651629) [vertex] {$f$};
            \node (g) at (-0.86602540378, 0.5) [vertex] {$g$};
            \node (h) at (-0.36602540378, 1.366025404) [vertexbc] {$h$};
            \node (i) at (0.5, 0.866025404) [vertex] {$i$};
            \node (j) at (1.36602540378, 1.366025404) [vertex] {$j$};
             \node (k) at (1.86602540378, 0.5) [vertexbc] {$k$};
            \node (l) at (2.17504239816,  -0.45105651629) [vertex] {$l$};
           
            \draw [-, line width=0.03cm] (a) to (b);
            \draw [-, line width=0.03cm] (a) to (e);
            \draw [-, line width=0.03cm] (b) to (c);
            \draw [-, line width=0.03cm] (c) to (d);
            \draw [-, line width=0.03cm] (d) to (e);
            \draw [-, line width=0.03cm] (c) to (i);
            \draw [-, line width=0.03cm] (b) to (i);
            \draw [-, line width=0.03cm] (f) to (g);
            \draw [-, line width=0.03cm] (g) to (h);
            \draw [-, line width=0.03cm] (j) to (k);
            \draw [-, line width=0.03cm] (k) to (l);
            \draw [-, line width=0.03cm] (g) to (b);
            \draw [-, line width=0.03cm] (c) to (k);
            \draw [-, line width=0.03cm] (h) to (i);
            \draw [-, line width=0.03cm] (f) to (a);
            \draw [-, line width=0.03cm] (d) to (l);
            \draw [-, line width=0.03cm] (i) to (j);
            
            \node (a1) at (5.79, -0.95105651629) [vertexbc] {4};
            \node (b1) at (6,0) [vertexa] {10};
            \node (c1) at (7,0) [vertex] {10};
            \node (d1) at (7.30901699437, -.95105651629) [vertex] {4};
            \node (e1) at (6.5, -1.39680224667) [vertex] {2};
            \node (f1) at (4.92,  -0.45105651629) [vertex] {2};
            \node (g1) at (5.23, 0.5) [vertex] {4};
            \node (h1) at (5.73, 1.366025404) [vertexbc] {3};
            \node (i1) at (6.5, 0.866025404) [vertex] {8};
            \node (j1) at (7.36602540378, 1.366025404) [vertex] {3};
             \node (k1) at (7.86602540378, 0.5) [vertexbc] {4};
            \node (l1) at (8.17504239816,  -0.45105651629) [vertex] {2};
           
            \draw [-, line width=0.03cm] (a1) to (b1);
            \draw [-, line width=0.03cm] (a1) to (e1);
            \draw [-, line width=0.03cm] (b1) to (c1);
            \draw [-, line width=0.03cm] (c1) to (d1);
            \draw [-, line width=0.03cm] (d1) to (e1);
            \draw [-, line width=0.03cm] (c1) to (i1);
            \draw [-, line width=0.03cm] (b1) to (i1);
            \draw [-, line width=0.03cm] (f1) to (g1);
            \draw [-, line width=0.03cm] (g1) to (h1);
            \draw [-, line width=0.03cm] (j1) to (k1);
            \draw [-, line width=0.03cm] (k1) to (l1);
            \draw [-, line width=0.03cm] (g1) to (b1);
            \draw [-, line width=0.03cm] (c1) to (k1);
            \draw [-, line width=0.03cm] (h1) to (i1);
            \draw [-, line width=0.03cm] (f1) to (a1);
            \draw [-, line width=0.03cm] (d1) to (l1);
            \draw [-, line width=0.03cm] (i1) to (j1);
            
        \end{tikzpicture}
        \end{center}
        \caption{Configuration $C_7$}
        \label{fig:conf4}
    \end{figure}
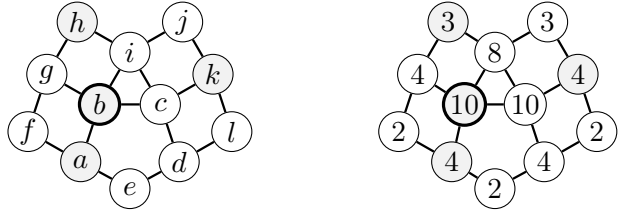

\begin{proof}
Assume that $G$ contains $C_7$. We follow the notation of Figure~\ref{fig:conf4}. Let $G'=G-b\in\mathcal{F}$, and assume that $G'$ admits a distance-$2$ $12$-coloring. Uncolor all vertices of $C_7$. For each uncolored vertex $v$, let $L(v)$ be the set of free colors of $v$ and let $\ell(v)$ be the minimum possible size of $L(v)$. Note that if two vertices are within distance 2 in $G$ using edges not depicted on Figure~\ref{fig:conf4}, then both their lists contain one more color than what is shown in the figure.

\begin{claim}
\label{claim:theclaim}
$L(e),L(f),L(l)$ are equal lists of size 2.
\end{claim}

\begin{proof}
Let $\alpha,\beta$ be two colors in $L(b)\setminus L(i)$. Let us first assume that $L(e) \ne \{ \alpha,\beta \}$. So we first can color $e$ with another color, and then we color $f,l$. Observe that afterwards, $L(b)\neq L(i)$ since $\alpha$ or $\beta$ still lies in $L(b)$ but not in $L(i)$. Now, we color $a,d,h,j,g,k,c$ in that order. Each time we color a vertex, either both $L(b)$ and $L(i)$ lose a color (in which case $L(b)$ still contains $\alpha$ or $\beta$ but not $L(i)$), or one of the two does not lose a color. Therefore, at the end, when only $b$ and $i$ remain to be colored, we have $L(b)$ and $L(i)$ both contain at least one color, and if they both contain only one, then they are different. We may thus color $b$ and $i$, which yields a coloring of $G$, a contradiction. 

Therefore, we must have $L(e)=\{\alpha,\beta\}$. If $L(f)\neq L(e)$, we can color $f$ with a color not in $L(e)$, then $e$ with $\alpha$. Now $L(b)\neq L(i)$ since $L(b)$ contains $\beta$ but not $L(i)$. We can then apply again the previous argument to color $G$ in the same order. Therefore, $L(f)=L(e)$, and by symmetry, we get the result.
\end{proof}

Let $H$ be the subgraph of $G^2$ induced by $\{a,\ldots,l\}$. Applying the algorithm $\A$ on $(H,\ell)$, we obtain that $H$ is $\ell$-choosable unless $L(v) \subseteq L(b)$ for all $v$ within distance $2$ from $b$. By symmetry, the same holds for $c$, so we may assume that $L(b)=L(c)$, and $L(b)$ contains all the other lists. 

Observe that if $e,h$ are within distance $2$, then $L(e)$ has size at least 3, which is incompatible with Claim~\ref{claim:theclaim}. Otherwise, if $L(e)\cap L(h)\neq\varnothing$, we can color $e$ and $h$ with the same color. In that case, $H-\{e,h\}$ has a happy coloring. Therefore, we may assume that $L(e)$ and $L(h)$ are disjoint. 

By Theorem~\ref{th:Csep}, $e$ and $k$ cannot be within distance $2$ in $G$. Moreover, if $L(e)\cap L(k)\neq\varnothing$, we may color them with the same color and $H-\{e,k\}$ is happy. Therefore we may assume that $L(e)\cap L(k)=\varnothing$.

Since $L(a) \cup L(h) \cup L(k) \subseteq L(b)$ and $\ell(b)=10$, there are two vertices $x,y\in\{a,h,k\}$ such that $L(x)\cap L(y)\neq\varnothing$. By Theorem~\ref{th:Csep}, $x,y$ cannot be within distance $2$. Therefore, we can color $x$ and $y$ with their common color, which does not appear on $L(e),L(f)$ or $L(l)$. Let $\ell'$ be the updated length of the lists, so that $\ell'(z)=\ell(z)$ for each $z\in\{e,f,l\}$. We can then check that $H-\{x,y\}$ is happy for $\ell'$. Therefore, $G$ admits a distance-$2$ $12$-coloring, which concludes the proof.
\end{proof}

\subsection{Configurations in $\mathcal{C}_{6^+}$}
\label{sec:C6+}

In order to prove that configurations of $\mathcal{C}_{6^+}$ are reducible, we actually show that each configuration of $\mathcal{C}_{6+}$ gives rise to a configuration of $\mathcal{C}$ which is in turn reducible, as shown before. In other words, our set of reducible configurations can actually can be restricted to only $\mathcal{C}$.

Consider a configuration of $\mathcal{C}_{6+}$.
By construction, this means that $G$ contains an edge $uv$ on a $6^+$-face $f$ such that more than $\frac13$ transits through $uv$ when applying the rules corresponding to \textbf{T} (see Figure~\ref{fig:larger}).  Let $f'$ be the face sharing the edge $uv$ with $f$.

We now consider several cases depending on which rules among \textbf{T1-6} made some weight transit through $uv$. 
\tikzstyle{3vertex}=[fill=black, circle,minimum size=4pt, inner sep=0pt]
\tikzstyle{vertex}=[circle,draw, minimum size=4pt, inner sep=0pt]
\begin{itemize}
\item Assume that \textbf{T1} applies, \emph{i.e.} by symmetry, $\deg(u)=3$. Note that, by Lemma~\ref{lem:config1}, the face $f'$ cannot be a triangle hence \textbf{T4} cannot apply. Since more than $\frac13$ transits through $uv$, another rule must be used. This creates one of the following configurations of $\mathcal{C}$: \begin{tikzpicture}[scale=.5]
\node[3vertex] (v) at (0,0) {};
\node[3vertex] (v') at (1,0) {};
\draw (v) -- (v');
\end{tikzpicture} (\textbf{T1}), \begin{tikzpicture}[scale=.3,baseline=-1mm]
\node[3vertex] (a) at (0,0) {};
\node[vertex] (b) at (1,0) {};
\node[vertex] (c) at (2,0) {};

\node[vertex] (bc) at (1.5,0.866) {};
\node[vertex] (bc') at (1.5,-.866) {};

\draw (a) -- (b) -- (bc') -- (c) -- (bc) -- (b) -- (c);
\end{tikzpicture} (\textbf{T2}), \begin{tikzpicture}[scale=.3,baseline=-1mm]
\node[vertex] (a) at (0,0) {};
\node[vertex] (b) at (1,0) {};
\node[3vertex] (c) at (2,0) {};

\node[vertex] (ab) at (.5,0.866) {};
\node[vertex] (a') at (0,-1) {};
\node[vertex] (b') at (1,-1) {};
\node[vertex] (c') at (2,-1) {};

\draw (a) -- (ab) -- (b) -- (a) -- (a') -- (b') -- (b) -- (c) -- (c') -- (b');
\end{tikzpicture} (\textbf{T3}), \begin{tikzpicture}[scale=.3,baseline=-1mm]
\node[3vertex] (b) at (1,0) {};
\node[vertex] (c) at (2,0) {};
\node[vertex] (d) at (3,0) {};
\node[vertex] (e) at (2,-1) {};
\node[vertex] (f) at (3,-1) {};

\node[vertex] (cd) at (2.5,0.866) {};
\node[vertex] (ef) at (2.5,-1.866) {};

\draw (b) -- (c) -- (cd) -- (d) -- (f) -- (ef) -- (e) -- (c) -- (d);
\draw (e) -- (f);
\end{tikzpicture} or \begin{tikzpicture}[scale=.3,baseline=-1mm]
\node[vertex] (a) at (0,0) {};
\node[vertex] (b) at (1,0) {};

\node[vertex] (ab) at (.5,0.866) {};
\node[vertex] (a') at (0,-1) {};
\node[3vertex] (b') at (1,-1) {};

\draw (a) -- (ab) -- (b) -- (a) -- (a') -- (b') -- (b);
\end{tikzpicture} (\textbf{T5}). 
\item Assume that \textbf{T2} applies. If $f'$ is a triangle or \textbf{T3} or \textbf{T4} applies, we obtain a triangle edge-incident to a triangle and a $4^-$-face, which is reduced by our heuristic. Therefore, $f'$ is not a triangle and only \textbf{T2} or \textbf{T5} may apply. This creates  \begin{tikzpicture}[scale=.3,baseline=-1mm]
\node[vertex] (z) at (-1,0) {};
\node[vertex] (a) at (0,0) {};
\node[vertex] (b) at (1,0) {};
\node[vertex] (c) at (2,0) {};

\node[vertex] (bc) at (1.5,0.866) {};
\node[vertex] (bc') at (1.5,-.866) {};
\node[vertex] (az) at (-.5,0.866) {};
\node[vertex] (az') at (-.5,-.866) {};

\draw  (z) -- (az) -- (a) -- (az') -- (z) -- (a) -- (b) -- (bc') -- (c) -- (bc) -- (b) -- (c);
\end{tikzpicture} (\textbf{T2}) or \begin{tikzpicture}[scale=.3,baseline=-1mm]
\node[vertex] (a) at (0,0) {};
\node[vertex] (b) at (1,0) {};
\node[vertex] (c) at (2,0) {};
\node[vertex] (d) at (3,0) {};
\node[vertex] (e) at (2,-1) {};
\node[vertex] (f) at (3,-1) {};

\node[vertex] (cd) at (2.5,0.866) {};
\node[vertex] (ef) at (2.5,-1.866) {};
\node[vertex] (ab) at (.5,0.866) {};
\node[vertex] (ab') at (.5,-0.866) {};

\draw (a) -- (ab) -- (b) -- (ab') -- (a) -- (b) -- (c) -- (cd) -- (d) -- (f) -- (ef) -- (e) -- (c) -- (d);
\draw (e) -- (f);
\end{tikzpicture} (\textbf{T5}).
\item  Assume that \textbf{T3} applies, hence by definition, \textbf{T4} does not apply. We now get configuration $C_1$ or \begin{tikzpicture}[scale=.3,baseline=-1mm]
\node[vertex] (a) at (0,0) {};
\node[vertex] (b) at (1,0) {};
\node[vertex] (c) at (2,0) {};
\node[vertex] (cd) at (2.5,0.866) {};
\node[vertex] (ab) at (.5,0.866) {};
\node[vertex] (bc) at (1.5,0.866) {};
\draw (ab) -- (a) -- (b) -- (ab) -- (bc) -- (cd) -- (c) -- (b) -- (bc);
\end{tikzpicture} if \textbf{T5} applies. Moreover, if \textbf{T3} applies a second time in such a way that more than $\frac13$ transits through $uv$, then we obtain configuration $C_3$.
\item Assume that \textbf{T4} applies. Then it cannot be applied twice, hence \textbf{T5} is applied, which creates \begin{tikzpicture}[scale=.3,baseline=-1mm]
\node[vertex] (a) at (0,0) {};
\node[vertex] (b) at (1,0) {};
\node[vertex] (c) at (2,0) {};
\node[vertex] (cd) at (2.5,0.866) {};
\node[vertex] (ab) at (.5,0.866) {};
\node[vertex] (bc) at (1.5,0.866) {};
\draw (ab) -- (a) -- (b) -- (ab) -- (bc) -- (cd) -- (c) -- (b) -- (bc);
\end{tikzpicture}.
\item Finally, assume that \textbf{T5} is applied twice. Then both $u$ and $v$ are incident to triangles (that are vertex-incident to $f$). In that case, only $2\times\frac 16$ transits through $uv$, contradicting the hypothesis.
\end{itemize}

\section{Further work}

A first natural question raised by our work is the extension of our techniques to distance-$2$ choosability of planar graphs with maximum degree $4$. Most of our proof technique still applies to this case (namely, the use of Linear Programming, and the heuristic that automatically reduces the configurations). However, when reducing separating cycles (Section~\ref{sec:sepcyc}), we permute some colors, which is typically forbidden in the choosability setting. Due to that, our reductions of $C_1$ to $C_7$ would need more involved arguments. No hope is yet lost for these configurations to stay reducible, but the task seems more painful since one would need to tackle the cases where the vertices we want to color alike are actually within distance $2$ in $G$, which might require to consider many cases. 

Our second long-term goal would consist in constructing a fully automatic system to prove coloring results using (local) discharging. This requires several important improvements of  our system we describe below. 

First, we need to improve the heuristic. We could use an exact algorithm, but this indeed would have a prohibitive cost on the performances (see~\cite{hartke2016chromatic}). However, in the literature, the configurations are most of the time reduced manually using only a few methods. Implementing them would probably give enough power to obtain reasonable results and would probably allow us to have a strong enough program to prove that reducible configurations of many articles are indeed reducible. As a more concrete example, one can first implement the technique used to reduce the configurations $C_1$ to $C_7$, where we show that some lists of colors should be disjoint (and obtain a contradiction with the inclusion matrix). 

A second aspect is the choice of the rule templates. While the literature contains a wealth of different types of rules, most of the time, charge is not moved over large distances. The next step is then probably to implement several classic rule templates and see what can be proved using them. 

Last, when our Linear Program returns a list of configurations such that one of them has to be shown to be reducible, we often show that a sub-configuration of this configuration is actually reducible. We recognize such configurations by experience. Reducing these sub-configurations permits to drastically fasten the convergence of our algorithm. It is not completely clear how many rounds would be needed if we only try to reduce configurations output by the Linear Program rather than sub-configurations. In a long time perspective, it might be interesting to see if learning algorithms can be used to detect sub-configurations that are likely to be reducible.

\bibliographystyle{plain}
\bibliography{biblio}

\end{document}